\documentclass[final,3p,authoryear]{elsarticle}

\journal{Transportation Research Part E}
\usepackage{lineno}
\modulolinenumbers[1]

\usepackage{amsmath, amssymb, amsthm, mathtools}
\usepackage{booktabs}
\usepackage{array}
\usepackage{multirow}
\usepackage{float}
\usepackage{graphicx}
\usepackage{xcolor}
\usepackage{adjustbox}     
\usepackage{threeparttable} 
\usepackage{enumitem}
\usepackage{setspace}
\usepackage{caption}
\usepackage{subcaption}
\usepackage{longtable}
\usepackage{tabularx}
\newcolumntype{C}{>{\centering\arraybackslash}X}
\usepackage{algorithm}
\usepackage{rotating}
\usepackage{algpseudocode}
\usepackage{tcolorbox}
\usepackage{makecell} 
\usepackage{xltabular} 
\usepackage{pdflscape}
\usepackage{xurl}
\usepackage{hyperref}
\hypersetup{
    colorlinks=true,
    linkcolor=blue,
    citecolor=blue,
    urlcolor=blue
}
\usepackage[nameinlink,capitalise,noabbrev]{cleveref}

\renewcommand{\arraystretch}{1.15}

\newtheorem{proposition}{Proposition}

\theoremstyle{definition}
\newtheorem{definition}{Definition}

\newtheorem{remark}{Remark}

\newcommand{\VMDIPF}{\mathrm{VMD}^{\mathrm{IPF}}}

\newcommand{\arc}[2]{(#1,#2)}

\begin{document}

\begin{frontmatter}

\title{Securing the Flow: Strategic Storage and Corridor Design for Maritime Energy Security under Chokepoint Disruptions}

\author[1]{Monit Sharma}
\author[1]{Hoong Chuin Lau\corref{cor1}}
\cortext[cor1]{Corresponding author}
\ead{hclau@smu.edu.sg}
\address[1]{School of Computing and Information Systems, Singapore Management University, Singapore}
\begin{abstract}
Maritime disruption planning requires joint decisions on strategic storage and
transport capacity before corridor conditions are known. We develop a
two-stage stochastic multi-commodity network-design model in which a national
energy-security planner selects corridor commitments, port--commodity
inventories, and optional storage expansion, then routes crude oil, LNG,
LPG, and fertilizer through scenario-dependent flows. The formulation
combines commodity admissibility, decision-dependent probabilities, corridor
concentration limits, and mean--CVaR risk. It admits an exact MILP
reformulation; a scenario-wise multi-cut Benders implementation serves as an
independent verification tool. In the 729-scenario verification,
the extensive form is faster than the Benders implementation.

In a data-anchored, stylized India-inspired case study, stochastic modelling
improves on the expected-value policy by about 1.35\%, and decision-dependent
probabilities are worth 0.67\% at the analyst-set exposure scale. Commodity
service is heterogeneous: crude and fertilizer remain fully served, while LNG
and LPG reach worst positive-probability service of 34.6\% and 33.3\%,
respectively. Under medium assumed capex, storage expansion adds 1,775
normalised capacity units
at seven predefined port--commodity pairs and reduces active corridor
commitments from 21 to 20. The calibrated binary Hormuz--Bab incidence
probabilities factorise exactly. Across 168 primary fixed-marginal cells,
dependence changes economic exposure without making the
independence-optimised policy suboptimal; selected capacity and commitment-cost
stresses change first-stage decisions, with maximum mechanism regret of 1.1673\%.
\end{abstract}

\begin{keyword}
Maritime energy security \sep Stochastic network design \sep
Decision-dependent uncertainty \sep CVaR \sep Chokepoint disruptions \sep
Strategic inventory
\end{keyword}

\end{frontmatter}

\section{Introduction}
\label{sec:introduction}

Capacitated network design under uncertainty is a foundational challenge in
operations research, with applications spanning telecommunications, freight
logistics, humanitarian relief, and energy supply chains. When disruptions can
disable network arcs, the planner faces a two-stage decision: invest in
infrastructure, backup capacity, and strategic inventory \emph{before}
uncertainty is resolved, then route commodities through the surviving network
\emph{after} a disruption scenario materialises. While the stochastic
capacitated multimodal network design problem has been studied
extensively, many formulations simplify either the joint structure of
multi-corridor disruptions or the relationship between strategic commitments
and the probability distribution used for planning \citep{snyder2005, lim2010}.

In this paper, we study the resilience of an import-dependent economy exposed
to maritime chokepoint disruptions. The model combines three practically
important features: dependence-aware joint-disruption scenarios,
commitment-dependent exposure, and tail-risk sensitivity. The purpose of the
dependence component is not to presume that every scenario library is
positively correlated. Instead, the analysis distinguishes binary
disruption incidence, conditional severity, and broader multivariate
co-movement, and then tests when changes in incidence dependence alter strategic
design decisions.

To capture the strategic structure of maritime energy security, the model uses
a dual-risk framework:
\begin{enumerate}
    \item \textbf{Dependence-aware environmental risk.} Scenarios encode joint
    changes in corridor capacities, transport costs, bypass availability, and
    commodity admissibility. Because the presence of joint-disruption states
    does not by itself establish statistical dependence, the case study
    diagnoses the dependence present in the calibrated library and supplements
    it with a fixed-marginal stress test over the feasible range of the joint
    Hormuz--Bab disruption probability.
    \item \textbf{Decision-dependent strategic exposure.} Through
    decision-dependent probabilities, a first-stage commitment to a vulnerable
    corridor---through a long-term contract, infrastructure activation, or
    dedicated capacity arrangement---changes the exposure-adjusted probability
    distribution used for planning. The mechanism is tied to ex ante strategic
    commitment, not to realised second-stage flow.
\end{enumerate}

\paragraph{Motivating Problem}
The Strait of Hormuz, Bab el-Mandeb, and the Suez Canal are critical passages
for global energy trade. Regional escalation can create joint disruption
states across these corridors, while Cape rerouting, pipeline bypasses, and
strategic inventories provide imperfect substitutes. For an import-dependent
economy such as India, the planning problem is therefore to select suppliers,
transport options, bypass infrastructure, and strategic stockholding before
the disruption state is known. This requires a model that captures
(i)~multi-commodity flows with commodity-specific admissibility,
(ii)~dependence-aware disruption scenarios and explicit dependence
diagnostics, (iii)~affine commitment-dependent probability distortion, and
(iv)~tail-risk protection through a mean--CVaR objective.

\paragraph{Contributions}
This paper makes the following contributions to stochastic network design.
\begin{enumerate}[leftmargin=2em, itemsep=0.3em]
\item \textbf{Integrative strategic formulation.} We present a two-stage
stochastic capacitated multi-commodity network-design model that jointly
represents corridor commitments, port--commodity inventory costs $h_{ic}$,
optional priced storage expansion, commodity-specific admissibility,
joint multi-corridor disruption states, decision-dependent probabilities,
corridor concentration caps, and mean--CVaR risk aversion. Scenario-dependent
flows, rather than the activation variables alone, determine feasible
post-disruption transport paths.
\item \textbf{Exact reformulation and implementation verification.} The
binary--continuous products admit an exact MILP reformulation
(\Cref{prop:milp-reform}), and the risk term preserves scenario-wise
separability (\Cref{prop:cvar-decomp}). We derive validated scenario-specific
bounds, verify extensive-form/Benders equivalence, and audit objective
accounting and physical feasibility through automated consistency gates. In the
729-scenario verification, the extensive form is faster than the
Benders implementation; decomposition is retained as an independent
verification route.
\item \textbf{Dependence diagnostics and policy-regime evidence.} We
distinguish binary disruption incidence, conditional severity, and broader
factor dependence. A matched factorial probability family varies
$q=P(H\cap B)$ while preserving incidence marginals, conditional severity
profiles, and auxiliary-factor marginals. The experiments identify both a
primary regime in which dependence changes exposure but not the optimal policy
and stressed regimes in which route or inventory decisions change.
\item \textbf{Decision-dependent exposure and tail calibration.} We quantify
the value of endogenous probabilities over the complete globally valid
exposure-sensitivity range and select the primary CVaR confidence level using a
predeclared tail-composition rule applied to an $\alpha$-independent reference
design.
\item \textbf{Operational and policy evidence.} The Indian case study reports
stochastic-program value, transport-feasibility audits, commodity-specific
service continuity, pre-positioning-cost sensitivity, and storage
construction-cost decomposition. The findings are tested under probability,
severity, admissibility, support, network-redundancy, and first-stage-cost
perturbations.
\end{enumerate}
\paragraph{Assumptions and institutional scope}
The planner is a coordinated national-level energy-security authority rather
than a single operating firm. Strategic reserve operators directly control
inventory, while route availability and storage expansion may be influenced
through infrastructure approvals, long-term contracts, subsidies, and
regulation. Private carriers and terminal operators execute the resulting
flows. The abstraction therefore represents coordinated policy authority, not
centralised control of every voyage. Individual vessel rotations, fleet
balance, berth allocation, and sailing schedules are outside the strategic
planning horizon. \ref{app:institutional} maps the decisions to the
Indian institutional setting and explains how the abstraction transfers to
other national energy-security systems.

Figure~\ref{fig:intro-framework} summarises the model's information structure,
separating nonanticipative strategic design, the disruption and exposure
mechanisms, and scenario-dependent operational recourse.

\begin{figure}[H]
\centering
\includegraphics[width=\textwidth]{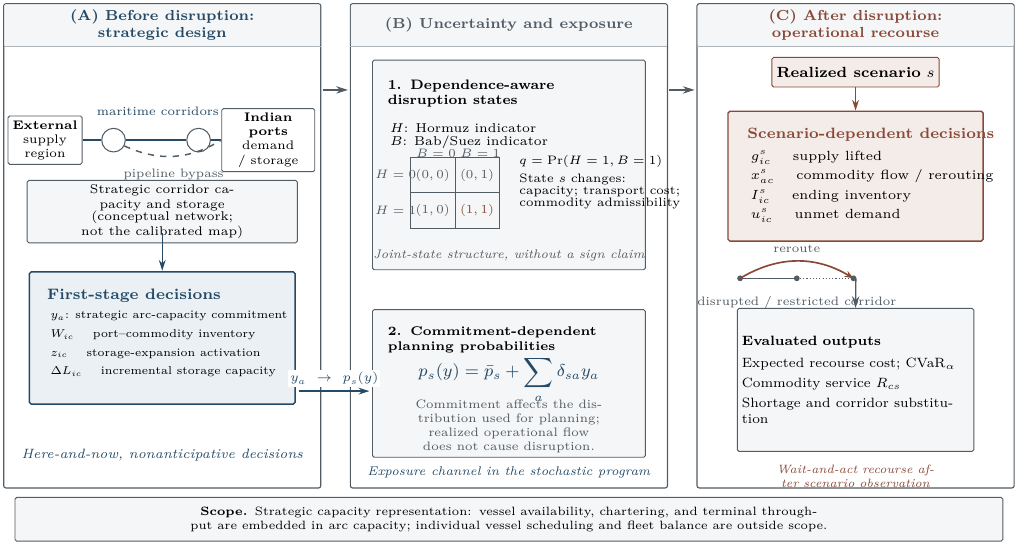}
\caption{Strategic maritime energy-planning framework. Before uncertainty is
observed, the planner commits strategic corridor capacity, positions
port--commodity inventory, and may expand storage. The uncertainty system
combines dependence-aware chokepoint states with commitment-dependent planning
probabilities. Once a scenario is realised, scenario-dependent supply lifting,
multi-commodity routing, inventory drawdown, and shortage decisions determine
cost, tail risk, and commodity-specific service continuity. Vessel and terminal
availability are represented through commodity-specific effective arc-capacity
envelopes rather than individual vessel schedules.}
\label{fig:intro-framework}
\end{figure}

\paragraph{Paper outline}
\Cref{sec:literature} reviews the related literature and
\Cref{sec:problem} describes the planning setting.
\Cref{sec:model} presents the formulation, while
\Cref{sec:structural,sec:algorithm} establish tractability and the exact
solution method. \Cref{sec:computational} reports the data-anchored, stylized India-inspired case study,
and \Cref{sec:computational-study} presents the bound-validation,
decomposition, and uncertainty-specification sensitivity evidence.
\Cref{sec:conclusion} concludes.

\section{Literature Review}
\label{sec:literature}

The paper lies at the intersection of four established research streams:
integrated location--inventory--transportation design, resilient supply-chain
network design, risk-averse and decision-dependent optimisation, and maritime
network planning. We review these streams comparatively because the contribution
is not a new method for each component in isolation. Rather, it is the integration
of strategic storage, transport commitments, multi-commodity recourse,
chokepoint-disruption dependence, decision-dependent exposure, and tail-risk
control in one national-level maritime planning model. Representative studies
and the boundary of the present contribution are summarised in
\Cref{tab:literature-comparison}.

\paragraph{Integrated location--inventory--transportation design}
The location--inventory literature established that facility configuration,
transport allocation, and inventory policy should be optimised jointly because
facility consolidation changes both risk-pooling benefits and transport costs.
Early formulations include the location--inventory models of
\citet{erlebacher2000,daskin2002,shen2003}. \citet{miranda2004} incorporate
stochastic demand, safety stock, and replenishment decisions into strategic
distribution-network design, while \citet{snyder2007} optimise location,
allocation, and inventory under scenario-based random parameters.
\citet{tancrez2012} extend this integration to a three-level network with
explicit flow allocation and shipment-size decisions. Later work adds richer
transportation structure: \citet{zheng2019} jointly optimise facility location,
inventory, and vehicle routing, and \citet{wang2022} incorporate third-party
logistics-provider selection. In the petroleum domain, \citet{xu2024petroleum}
optimise inventory and transportation under distributional ambiguity.

These studies directly address the integration central to our model, but
most concern firm-level distribution systems, demand or lead-time uncertainty,
and warehouse-to-customer allocation or routing. They generally do not model
joint maritime chokepoint failures, national strategic reserves,
commitment-dependent disruption probabilities, or a multi-commodity
mean--CVaR objective. Our model therefore builds on, rather than replaces, the
location--inventory--transportation literature.

\paragraph{Resilient supply-chain network design and recovery}
Reliable facility-location models protect networks against facility failures
through redundancy, assignment flexibility, or protection decisions
\citep{snyder2005,lim2010}. \citet{klibi2012} study a capacitated
location--transportation problem under multi-hazard disruptions and compare
alternative response policies. \citet{baghalian2013} combine multi-product
facility location, inventory, shipment, and service-level decisions under demand
and supply disruptions. \citet{fattahi2017} integrate operational and disruption
risk with lead-time-sensitive demand, and \citet{fattahi2020} introduce an
explicit recovery-loss resilience metric in a two-stage stochastic network
design model. \citet{alikhani2021} jointly consider fortification, inventory
pre-positioning, reserved capacity, inventory sharing, and direct delivery.
Related endogenous-risk models let protection or design decisions affect
post-disruption capacity or demand \citep{bhuiyan2020,luo2025two}.

This stream clarifies the resilience interpretation of the present model. Our
primary measure is service continuity and feasible rerouting after a disruption,
reported through expected and worst-case commodity service, shortage, and
corridor substitution. The model does not estimate a continuous time-to-recovery
trajectory; instead, it is a strategic two-stage design model in which
scenario-contingent flows, supply lifting, inventory drawdown, and shortage are
recourse actions.

\paragraph{Risk aversion and decision-dependent uncertainty}
The linear auxiliary-variable representation of CVaR follows
\citet{rockafellar2000}. Risk-averse stochastic programs and decomposition have
been developed for disaster management, facility location, hub location, and
supply-chain disruption mitigation \citep{noyan2012,carneiro2010risk,
ghaffarinasab2023}. \citet{gao2019} link strategic inventory positioning to
simultaneous disruptions through a worst-case CVaR exposure model, but in a
single-commodity setting with exogenous risk.

Decision-dependent uncertainty distinguishes decisions that alter a probability
distribution from decisions that alter information timing
\citep{goel2006}. Affine decision-dependent probabilities are formalised by
\citet{hellemo2018}; related facility and network-design models address
endogenous demand, protection effectiveness, or adaptive uncertainty sets
\citep{basciftci2021,bhuiyan2020,zhao2019,luo2025two}. Our formulation uses
this machinery narrowly: arc activation represents an ex ante strategic
commitment, and the coefficients $\delta_{sa}$ shift the planning distribution
rather than asserting that realised recourse flow causes a disruption. The
coefficients are treated as analyst-set exposure parameters and are subjected
to a complete probability-validity and magnitude sensitivity analysis.

\paragraph{Maritime network planning and energy supply chains}
Maritime operations research provides detailed carrier-level models of ship
services, fleet deployment, schedules, and cargo paths. \citet{agarwal2008}
optimise ship schedules and cargo routing jointly; \citet{brouer2014} provide a
benchmark formulation for liner-shipping network design; and \citet{koza2020}
integrate service scheduling with network design. Operational disruption
recovery models, such as \citet{li2015}, use speed changes, port skipping, and
port swapping to restore vessel schedules.

Those models answer a different planning question from ours. We do not optimise
individual vessel rotations, weekly service frequencies, or fleet balance.
Vessel and contracting availability are aggregated into scenario-dependent arc
capacities, while the national-level planner selects strategic inventory and
transport-capacity commitments before routing commodities through the surviving
network. This distinction makes the model a strategic capacitated network-flow
model rather than a tactical vessel-scheduling model.

Oil and petroleum supply-chain models provide a closer commodity context.
\citet{ribas2010} optimise a multimodal integrated oil chain under production,
demand, and price uncertainty; \citet{carneiro2010risk} apply mean--CVaR to
oil-supply-chain planning; \citet{lima2018} develop multistage stochastic
tactical planning for a downstream oil network; and \citet{xu2024petroleum}
combine petroleum inventory and transportation under distributional ambiguity.
These studies integrate production, storage, and transport, but do not focus on
joint geopolitical chokepoint failures or commitment-dependent scenario
probabilities. At the system level, \citet{verschuur2025} quantify trade impacts
of chokepoint disruption using AIS-based simulation, providing an empirical
counterpart but not a prescriptive storage-and-corridor design model.

\begin{landscape}
\begin{table}[H]
\centering
\caption{Representative literature and the positioning of the present study.}
\label{tab:literature-comparison}
\footnotesize
\setlength{\tabcolsep}{4pt}
\renewcommand{\arraystretch}{1.12}
\begin{tabularx}{\linewidth}{@{}
>{\raggedright\arraybackslash}p{3.0cm}
>{\raggedright\arraybackslash}X
>{\raggedright\arraybackslash}X
>{\raggedright\arraybackslash}X
>{\raggedright\arraybackslash}p{3.0cm}
>{\raggedright\arraybackslash}p{2.8cm}@{}}
\toprule
Study & Strategic decisions & Transportation representation & Uncertainty and resilience & Risk/endogeneity & Primary domain \\
\midrule
\citet{daskin2002,shen2003} & Facility location, allocation, safety stock & Customer--DC allocation and transport cost & Demand variability; risk pooling & Expected cost & Distribution networks \\
\citet{miranda2004,snyder2007} & Facility location and inventory policy & Assignment and shipment allocation & Stochastic demand or scenario parameters & Expected cost & Distribution networks \\
\citet{tancrez2012} & DC location, inventory, shipment size & Three-level commodity flows & Deterministic strategic design & Expected cost & Multi-echelon supply chain \\
\citet{zheng2019,wang2022} & Location, inventory, routing or 3PL choice & Vehicle routes or logistics-provider assignments & Demand and operational uncertainty & Expected cost & Firm-level supply chains \\
\citet{klibi2012,baghalian2013} & Facility/capacity design and inventory & Location--transport allocation and shipments & Disruptions, response policies, service levels & Stochastic/robust criteria & Resilient supply networks \\
\citet{fattahi2017,fattahi2020,alikhani2021} & Location, inventory, capacity, fortification & Allocation, direct delivery, reserved capacity & Disruption response and recovery & Resilience metrics & Resilient supply chains \\
\citet{bhuiyan2020,luo2025two} & Network protection or facility design & Network allocation/flow & Endogenous and exogenous uncertainty & Decision-dependent risk & General networks \\
\citet{ribas2010,lima2018,xu2024petroleum} & Production, storage, and inventory planning & Road, rail, pipeline, maritime or depot transport & Demand, price, production, or distributional uncertainty & Stochastic, robust, or DRO & Oil and petroleum logistics \\
\citet{agarwal2008,brouer2014,koza2020,li2015} & Service network, fleet deployment, schedule recovery & Vessel rotations, schedules, and cargo paths & Operational delays and service recovery & Expected operational cost & Liner shipping \\
\citet{verschuur2025} & No design decision & AIS-based rerouting simulation & Chokepoint disruption impact & Descriptive simulation & Global maritime trade \\
\addlinespace
\textbf{Present study} & \textbf{Inventory positioning and expansion activation at predefined port--commodity pairs; corridor commitments} & \textbf{Capacitated multi-commodity scenario flows} & \textbf{Joint chokepoint states, service continuity, support and severity perturbations} & \textbf{Mean--CVaR and decision-dependent probabilities} & \textbf{National maritime energy security} \\
\bottomrule
\end{tabularx}
\begin{minipage}{0.98\linewidth}
\footnotesize\emph{Note:} The table is intentionally comparative rather than
exhaustive. ``Transportation representation'' distinguishes aggregate
allocation/flow models from vehicle- or vessel-scheduling models.
\end{minipage}
\end{table}
\end{landscape}

\paragraph{Position relative to the existing literature}
The comparison shows that the research gap is not the absence of
location--inventory--transportation models, resilient network-design models, or
maritime routing models. Each is well established. The narrower gap addressed
here is a strategic national-level model that jointly considers
(i)~inventory positioning and expansion activation at predefined
port--commodity pairs,
(ii)~multi-commodity transport commitments and scenario-contingent flows,
(iii)~joint multi-corridor chokepoint disruptions with explicit dependence
diagnostics,
(iv)~commitment-dependent planning probabilities, and
(v)~mean--CVaR tail protection. The MILP reformulation and Benders method use
standard optimisation principles; the incremental contribution is their
validated integration with the maritime-resilience setting, together with the
fixed-marginal dependence experiment and the Indian case-study evidence.

\section{Problem Description}
\label{sec:problem}

We consider a directed network $G=(N,A)$, where $N$ is the set of physical
nodes and $A\subseteq N\times N$ is the set of feasible directed transport
links. The node-role subsets $N^{\mathrm S}$, $N^{\mathrm D}$, and
$N^{\mathrm K}$ identify supply, demand, and storage-capable nodes,
respectively. A physical node may have more than one role; for example, an
Indian import terminal can simultaneously be a demand node and a strategic
storage site. Multiple commodities $c\in C$ are transported over arcs grouped
into corridors $\ell\in\mathcal L$ that share a geographic bottleneck.

The planning problem has two stages:
\begin{enumerate}[leftmargin=2em]
    \item \textbf{First stage (strategic design).} Before the disruption is
    known, the planner selects arc commitments $y_a$, inventory positions
    $W_{ic}$, and, when the expansion extension is enabled, expansion-activation
    and incremental-capacity decisions $(z_{ic},\Delta L_{ic})$. Arc commitments
    represent long-term contracts, infrastructure enablement, or dedicated
    strategic capacity commitment rather than a realised voyage.
    \item \textbf{Second stage (operational recourse).} After scenario $s\in S$
    is observed, the planner chooses supply lifting, multi-commodity flows,
    inventory drawdown, and unmet demand. These scenario-dependent variables
    determine how the fixed strategic design is used under the surviving
    network.
\end{enumerate}

\paragraph{Strategic transportation representation and model boundary}
The binary variable $y_a$ does not constitute a route by itself. It makes
strategic transport capacity on arc $a$ available. Feasible shipping paths are
formed by the scenario-dependent variables $x_{ac}^{s}$, which satisfy nodal
material conservation, commodity admissibility, activated capacity, and
corridor-concentration constraints. Available fleet, charter, terminal, and
pipeline throughput are represented by the commodity-specific effective
capacity envelope $\bar u_{ac}^{s}$. Modelling vessel
classes, vessel rotations, sailing times, port-call sequences, and fleet
balance would require a tactical liner-shipping or fleet-deployment model and
is outside the present strategic horizon.

\paragraph{Nonanticipativity}
The strategic variables $y_a$, $W_{ic}$, $z_{ic}$, and $\Delta L_{ic}$ are
selected before the scenario is known and are therefore common across
scenarios. Equivalently, scenario-indexed copies could be introduced with
constraints $y_a^s=y_a^{s'}$ and $W_{ic}^s=W_{ic}^{s'}$ for all $s,s'$, but
those copies are redundant. Allowing unrestricted scenario-dependent strategic
decisions would instead produce a perfect-information wait-and-see model.

In the case study, the planner represents the coordinated apparatus through
which a national government manages maritime energy security: strategic
reserve operators, pipeline-offtake authorities, regulators, and crisis
response bodies. \ref{app:institutional} provides the detailed Indian
institutional mapping.

Three structural features distinguish the uncertainty and risk representation:
\begin{itemize}[leftmargin=2em, itemsep=0.2em]
    \item \textbf{Dependence-aware joint disruption states.} Scenarios encode
    simultaneous changes in corridor capacity, transport cost, bypass
    availability, and commodity-specific admissibility. The calibrated
    incidence structure is diagnosed explicitly, and alternative dependence
    levels are introduced through a matched fixed-marginal stress-test family.
    \item \textbf{Decision-dependent probabilities.} Strategic commitments to
    vulnerable corridors shift the exposure-adjusted scenario probabilities
    $p_s(\mathbf y)$.
    \item \textbf{Tail-risk aversion.} A mean--CVaR objective protects against
    high-cost disruption states in addition to expected recourse cost.
\end{itemize}

\section{Mathematical Formulation}
\label{sec:model}

We present the formulation in three layers: the base two-stage stochastic
program (\Cref{sec:base-model}), the decision-dependent uncertainty probability mechanism
(\Cref{sec:ddu}), and the CVaR integration (\Cref{sec:cvar}).

\subsection{Sets, Parameters, and Decision Variables}
\label{sec:notation}

\begin{table}[H]
\centering
\caption{Sets and indices.}
\label{tab:sets}
\begin{tabularx}{\textwidth}{@{}>{\raggedright\arraybackslash}p{3.0cm}X@{}}
\toprule
Symbol & Description \\
\midrule
$N$ & Set of physical nodes ($|N|=16$ in the case study) \\
$A\subseteq N\times N$ & Set of feasible directed arcs ($|A|=28$) \\
$C$ & Commodities: crude oil, LNG, LPG, and fertilizer \\
$S$ & Disruption scenarios ($|S|=9$ in the calibrated library) \\
$\mathcal L$ & Corridors: Hormuz, Bab/Suez, Cape, pipeline, and direct \\
$N^{\mathrm S},N^{\mathrm D},N^{\mathrm K}\subseteq N$ & Supply, demand, and storage-capable node sets \\
$\mathcal R\subseteq N^{\mathrm S}\times C$ & Node--commodity pairs with positive available supply \\
$\mathcal D\subseteq N^{\mathrm D}\times C$ & Node--commodity pairs with positive demand \\
$\mathcal K\subseteq N^{\mathrm K}\times C$ & Node--commodity pairs with installed storage capacity \\
$\mathcal K^{+}\subseteq\mathcal K$ & Predefined port--commodity pairs eligible for incremental expansion \\
\bottomrule
\end{tabularx}
\end{table}

\begin{table}[H]
\centering
\caption{Model parameters.}
\label{tab:params}
\begin{tabularx}{\textwidth}{@{}>{\raggedright\arraybackslash}p{3.0cm}X@{}}
\toprule
Symbol & Description \\
\midrule
$\bar p_s$ & Baseline probability of scenario $s$, $\sum_s\bar p_s=1$ \\
$\delta_{sa}$ & Analyst-set DDU probability adjustment associated with activation of arc $a$ \\
$f_a$ & Fixed strategic commitment cost for arc $a$ \\
$c_{ac}^{s}$ & Transport cost on arc $a$, commodity $c$, scenario $s$ \\
$\bar u_{ac}^{s}$ & Available capacity of arc $a$ for commodity $c$ in scenario $s$ \\
$\theta_{ac}^{s}\in\{0,1\}$ & Scenario-specific commodity admissibility \\
$d_{ic}$ & Demand at $(i,c)\in\mathcal D$ \\
$r_{ic}$ & Maximum available supply at $(i,c)\in\mathcal R$ \\
$L_{ic}^{0}$ & Existing installed storage capacity at $(i,c)\in\mathcal K$ \\
$\overline{\Delta L}_{ic}$ & Maximum incremental capacity at $(i,c)\in\mathcal K^{+}$ \\
$h_{ic}$ & Holding cost at port--commodity pair $(i,c)\in\mathcal K$ \\
$F_{ic}^{L}$ & Fixed construction and expansion-activation cost \\
$v_{ic}^{L}$ & Variable construction cost per unit of added storage capacity \\
$\pi_c$ & Unmet-demand penalty for commodity $c$ \\
$\beta_{\rm end}\in[0,1]$ & Terminal-inventory cost multiplier; primary specification $\beta_{\rm end}=0$ \\
$\bar\Phi_\ell$ & Corridor-concentration cap for corridor $\ell$ \\
$M_s^Q,M_s^\xi$ & Valid scenario-specific upper bounds for $Q_s$ and $\xi_s$ \\
$\lambda\in[0,1]$ & Weight on CVaR in the mean--CVaR objective \\
$\alpha\in(0,1)$ & CVaR confidence level \\
\bottomrule
\end{tabularx}
\end{table}

\begin{table}[H]
\centering
\caption{Decision variables.}
\label{tab:variables}
\begin{tabularx}{\textwidth}{@{}>{\raggedright\arraybackslash}p{3.0cm}X@{}}
\toprule
Symbol & Description \\
\midrule
\multicolumn{2}{@{}l}{\emph{First-stage variables}} \\
$y_a\in\{0,1\}$ & $=1$ if strategic capacity on arc $a$ is committed \\
$W_{ic}\ge0$ & Strategic inventory pre-positioned at $(i,c)\in\mathcal K$ \\
$z_{ic}\in\{0,1\}$ & $=1$ if incremental storage expansion is activated at predefined pair $(i,c)\in\mathcal K^{+}$ \\
$\Delta L_{ic}\ge0$ & Incremental storage capacity added at $(i,c)\in\mathcal K^{+}$ \\[0.5em]
\multicolumn{2}{@{}l}{\emph{Second-stage variables, indexed by scenario $s$}} \\
$x_{ac}^{s}\ge0$ & Flow of commodity $c$ on arc $a$ \\
$g_{ic}^{s}\ge0$ & Actual supply lifted at $(i,c)\in\mathcal R$ \\
$I_{ic}^{s}\ge0$ & Ending inventory at $(i,c)\in\mathcal K$ \\
$u_{ic}^{s}\ge0$ & Unmet demand at $(i,c)\in\mathcal D$ \\
$Q_s\ge0$ & Equality-defined scenario recourse-cost variable \\[0.5em]
\multicolumn{2}{@{}l}{\emph{Risk variables}} \\
$\nu\ge0$ & VaR auxiliary variable \\
$\xi_s\ge0$ & CVaR excess variable \\
\bottomrule
\end{tabularx}
\end{table}

\subsection{Base Two-Stage Stochastic Model}
\label{sec:base-model}

We first present a baseline two-stage stochastic program under two
simplifying assumptions: (i) scenario probabilities are exogenous and
therefore fixed at their baseline values $\bar p_s$, and (ii) the
planner is risk-neutral, so only expected cost is minimised. This base
model provides the benchmark formulation to which we later add
decision-dependent probabilities and mean--CVaR risk aversion.

In the first stage, before the disruption scenario is realised, the
planner chooses arc-commitment decisions $y_a$, strategic inventory
positions $W_{ic}$, and optional expansion decisions $(z_{ic},\Delta L_{ic})$. In the second stage, after scenario $s \in S$ is
revealed, the planner routes flows through the surviving network,
draws down inventory, and allows unmet demand at a penalty. The second
stage is therefore a recourse problem that evaluates how costly the
first-stage design is under each disruption realisation.

\paragraph{Objective}
The optional incremental-storage construction cost is
\begin{align}
C^{L}(\mathbf z,\Delta\mathbf L)
=
\sum_{(i,c)\in\mathcal K^{+}}
\left(F_{ic}^{L}z_{ic}+v_{ic}^{L}\Delta L_{ic}\right),
\label{eq:expansion-cost}
\end{align}
and is set to zero when expansion is disabled. Under exogenous probabilities
and risk neutrality, the objective is
\begin{align}
\min \; Z
&=
\underbrace{\sum_{a \in A} f_a y_a
+ \sum_{(i,c)\in\mathcal K} h_{ic} W_{ic}
+ C^{L}(\mathbf z,\Delta\mathbf L)}_{\text{first-stage design cost}}
+
\underbrace{\sum_{s \in S} \bar p_s \, Q_s(\mathbf y,\mathbf W)}_{\text{expected second-stage recourse cost}}.
\label{eq:obj-rn}
\end{align}

The first term captures the here-and-now design cost: $f_a y_a$ is the
fixed cost of committing strategic capacity on arc $a$, $h_{ic}W_{ic}$ is
the port--commodity inventory cost, and $C^{L}$ charges only incremental
storage capacity. Existing installed capacity is treated as sunk. The final
term is the
expected recourse cost, obtained by weighting the scenario-specific
second-stage cost $Q_s(\mathbf y,\mathbf W)$ by the exogenous scenario
probability $\bar p_s$. When expansion is enabled, the recourse feasible set
also depends on $\Delta\mathbf L$ through the storage-capacity constraints.

For each scenario $s$, the recourse function is defined as
\begin{align}
Q_s(\mathbf y,\mathbf W)
=
\min \;
& \sum_{a \in A}\sum_{c \in C} c_{ac}^s x_{ac}^s
+ \beta_{\rm end}\sum_{(i,c)\in\mathcal K} h_{ic} I_{ic}^s
+ \sum_{(i,c)\in\mathcal D} \pi_c u_{ic}^s .
\label{eq:recourse}
\end{align}
The primary single-period specification sets $\beta_{\rm end}=0$, so terminal
inventory has no salvage or terminal-valuation charge; $\beta_{\rm end}=1$ is
reported only as an economic sensitivity. The model also contains an explicit equality-defined variable $Q_s$ constrained
to equal the objective expression in \eqref{eq:recourse}. This distinction is
important for the McCormick reformulation: the upper bound must be valid for
every feasible value of the equality-defined variable, rather than merely for
the minimum cost of one feasible recourse policy.

The $\beta_{\rm end}\in\{0,0.5,1\}$ sensitivity gives the same canonical
solution because ending inventory is zero in every canonical scenario.

The three terms represent transportation cost, port--commodity holding cost at
storage pairs, and shortage penalties at genuine demand pairs. Supply is not
forced into the network. Instead, the recourse model chooses the amount
$g_{ic}^s$ lifted from each available source, up to the physical availability
$r_{ic}$.

\paragraph{Flow balance and physical variable domains}
For every node $i\in N$, commodity $c\in C$, and scenario $s\in S$,
material conservation is enforced by
\begin{align}
W_{ic}
+ g_{ic}^{s}
+ \sum_{\arc{j}{i}\in A} x_{ji,c}^{s}
- \sum_{\arc{i}{j}\in A} x_{ij,c}^{s}
- d_{ic}
+ u_{ic}^{s}
=
I_{ic}^{s},
\label{eq:flow-balance}
\end{align}
where absent pair-specific quantities are interpreted as zero. The production
formulation imposes
\begin{align}
0\le g_{ic}^{s} &\le r_{ic}, && (i,c)\in\mathcal R,\ s\in S,
\label{eq:supply-lifting}\\
0\le u_{ic}^{s} &\le d_{ic}, && (i,c)\in\mathcal D,\ s\in S,
\label{eq:shortage-bound}\\
W_{ic}=I_{ic}^{s}&=0, && (i,c)\notin\mathcal K,\ s\in S.
\label{eq:nonstorage-zero}
\end{align}
Thus, $r_{ic}$ is maximum available supply rather than mandatory injection,
$u_{ic}^{s}$ exists only where demand is present and cannot exceed that demand,
and end-of-period inventory is permitted only at installed storage pairs. These
restrictions prevent shortage or non-storage inventory from acting as an
unpriced artificial source or sink. Complete recourse is retained: for any
feasible first-stage inventory, choosing $g=0$, $x=0$, $u=d$, and $I=W$ is
feasible on storage pairs, with $I=0$ elsewhere.

\paragraph{Activated, admissible capacity}
For each arc $a\in A$, commodity $c\in C$, and scenario $s\in S$, flow is
bounded by the capacity that survives the scenario and can occur only if the
arc is committed and the commodity is admissible:
\begin{align}
0\le x_{ac}^{s}
\le \theta_{ac}^{s}\bar u_{ac}^{s}y_a,
\qquad a\in A,\ c\in C,\ s\in S.
\label{eq:admiss-tight}
\end{align}
If $y_a=0$ or $\theta_{ac}^{s}=0$, the corresponding flow is zero. Otherwise,
the flow is limited by the scenario-specific commodity-specific effective
capacity envelope. This single
constraint is the linking relation between strategic commitment and
scenario-contingent transportation; no weaker duplicate capacity constraint is
required.

\paragraph{Storage capacity and priced expansion}
With existing installed capacity only,
\begin{align}
0\le W_{ic}\le L_{ic}^{0},\qquad
0\le I_{ic}^{s}\le L_{ic}^{0},
\quad (i,c)\in\mathcal K,\ s\in S.
\label{eq:storage}
\end{align}
When expansion is enabled, the following constraints replace
\eqref{eq:storage} for $(i,c)\in\mathcal K^{+}$:
\begin{align}
0\le \Delta L_{ic}
&\le \overline{\Delta L}_{ic}z_{ic},
\label{eq:expansion-gate}\\
0\le W_{ic},\ I_{ic}^{s}
&\le L_{ic}^{0}+\Delta L_{ic},
\qquad s\in S.
\label{eq:expanded-storage}
\end{align}
The cost in \eqref{eq:expansion-cost} separates fixed construction from
variable capacity cost. Existing capacity $L_{ic}^{0}$ is sunk; only
$\Delta L_{ic}$ is charged. Pairs outside $\mathcal K$ retain the zero-domain
restrictions in \eqref{eq:nonstorage-zero}.
Because $\mathcal K^{+}\subseteq\mathcal K$, expansion is restricted to
predefined port--commodity pairs with installed storage capacity; the model
does not select unrestricted greenfield facility locations.

\paragraph{Corridor-concentration cap}
To avoid excessive concentration of traffic on a single vulnerable
corridor, we impose a corridor-concentration cap. For each corridor
$\ell \in \mathcal L$ and scenario $s \in S$,
\begin{align}
\sum_{a \in \ell} x_{ac}^{s}
\le
\bar\Phi_\ell
\sum_{a \in A} x_{ac}^{s},
\quad \forall \ell \in \mathcal L,\; c \in C,\; s \in S.
\label{eq:corridor-cap}
\end{align}

Constraint \eqref{eq:corridor-cap} is applied separately to each commodity.
The denominator contains only commodity $c$, so the same policy cap is
invariant to commodity-specific rescaling. When total flow of commodity $c$ is
zero, both sides are zero and the constraint is vacuously satisfied. Thus, no
division by total flow is performed and the constraint remains well-defined
under complete network failure.
Together, the objective and constraints define the baseline risk-neutral two-stage stochastic
model with fixed scenario probabilities. In the next subsection, we relax the exogeneity
assumption by allowing scenario probabilities to depend on first-stage
design choices.

\subsection{Decision-Dependent Disruption Probabilities}
\label{sec:ddu}

The baseline model assumes that the disruption probabilities
$\bar p_s$ are exogenous and fixed in advance. In many resilience-planning
settings, however, this assumption can be restrictive when alternative
first-stage designs imply different assessments of exposure to disruption
scenarios. We therefore allow the planning probabilities to depend on
first-stage arc-activation decisions. This is a conceptual exposure mechanism,
not a claim that an activation physically causes a geopolitical disruption.

This modelling step should be interpreted as an \emph{exposure-adjusted}
probability mechanism. Within the model, activating a corridor may shift
planning probability mass toward or away from scenarios involving that
corridor, thereby representing endogenous exposure assessments. The
coefficients are analyst-set exposure adjustments that may, in a future
empirical implementation, be informed by intelligence assessments, insurance
premia, transit data, or event studies. They are not estimates produced by any
named agency and should not be interpreted as physical causal effects of arc
activation.

\begin{definition}[Affine decision-dependent probabilities]
\label{def:ddp}
For each scenario $s \in S$, we let the probability of occurrence depend
affinely on the first-stage arc activation vector
$\mathbf y \in \{0,1\}^{|A|}$ according to
\begin{align}
p_s(\mathbf y)
=
\bar p_s + \sum_{a \in A} \delta_{sa}\, y_a,
\qquad \forall s \in S.
\label{eq:ddp}
\end{align}
Here, $\bar p_s$ denotes the baseline probability of scenario $s$, and
$\delta_{sa}$ is an analyst-set sensitivity coefficient representing the
planning-probability adjustment associated with activating arc $a$.
\end{definition}

The dependence of $p_s(\mathbf y)$ on the first-stage activation variables
$y_a$, rather than realised second-stage flows, is intentional. In the present
setting, $y_a$ represents a strategic commitment to a corridor through
contracting, infrastructure activation, or long-term routing dependence. The
resulting probability distortion therefore captures exposure created by the
design itself, rather than by scenario-contingent recourse adjustments. This
choice preserves the two-stage timing structure and keeps the endogenous
uncertainty model computationally tractable.

To ensure that \eqref{eq:ddp} defines a valid probability distribution for
every feasible first-stage decision vector, we impose
\begin{align}
p_s(\mathbf y) \ge 0,
\qquad \forall s \in S,\; \forall \mathbf y \in \{0,1\}^{|A|},
\label{eq:ddp-nonneg}
\end{align}
together with the mass-preservation condition
\begin{align}
\sum_{s \in S} \delta_{sa} = 0,
\qquad \forall a \in A.
\label{eq:ddp-mass}
\end{align}
Condition \eqref{eq:ddp-mass} ensures that activating arc $a$ redistributes
probability mass across scenarios without changing the total mass. Indeed,
summing \eqref{eq:ddp} over all scenarios gives
\[
\sum_{s \in S} p_s(\mathbf y)
=
\sum_{s \in S} \bar p_s
+
\sum_{a \in A} y_a \sum_{s \in S} \delta_{sa}
=
1,
\]
provided that the baseline probabilities satisfy
$\sum_{s \in S} \bar p_s = 1$.

Under \eqref{eq:ddp}, the expected recourse term in the objective becomes
\begin{align}
\sum_{s \in S} p_s(\mathbf y)\, Q_s(\mathbf y,\mathbf W)
&=
\sum_{s \in S}
\left(
\bar p_s + \sum_{a \in A} \delta_{sa} y_a
\right) Q_s(\mathbf y,\mathbf W) \notag \\
&=
\sum_{s \in S} \bar p_s Q_s(\mathbf y,\mathbf W)
+
\sum_{s \in S}\sum_{a \in A} \delta_{sa}\, y_a Q_s(\mathbf y,\mathbf W).
\label{eq:ddp-expand}
\end{align}
Thus, affine decision-dependent probabilities introduce bilinear products
of the form $y_a Q_s$. Analogous bilinear terms arise in the mean--CVaR
extension through products of the form $y_a \xi_s$. Since $y_a$ is binary
and both $Q_s$ and $\xi_s$ are bounded continuous variables, these
products admit an exact mixed-integer linear reformulation via McCormick
linearisation, as established in \Cref{prop:milp-reform}.

The affine specification \eqref{eq:ddp} is attractive for two reasons.
First, it is interpretable within the model: the coefficients $\delta_{sa}$
encode how individual arc activations adjust assessed exposure toward or away
from each scenario. Second, it preserves tractability, since the resulting
nonlinearities remain binary--continuous bilinear terms that can be
linearised exactly.

\subsection{Mean--CVaR Objective}
\label{sec:cvar}

The baseline model in \Cref{sec:base-model} is risk-neutral: it minimises
first-stage design cost plus the expected second-stage recourse cost.
While appropriate when average performance is the only concern, such an
objective may underweight rare but severe disruption outcomes. In a
maritime energy-security setting, however, the planner is typically
concerned not only with expected cost, but also with protection against
high-impact tail events such as major chokepoint closures or joint
multi-corridor disruptions. To capture this tradeoff, we replace the
risk-neutral objective by a mean--CVaR objective.

Specifically, we minimise a convex combination of expected recourse cost
and the conditional value-at-risk (CVaR) of recourse cost. The resulting
objective is
\begin{align}
\min \; Z^{\mathrm{CVaR}}
&=
\sum_{a \in A} f_a y_a
+
\sum_{(i,c)\in\mathcal K} h_{ic} W_{ic}
+
C^{L}(\mathbf z,\Delta\mathbf L)
+
(1-\lambda)\sum_{s \in S} p_s(\mathbf y)\, Q_s
+
\lambda \left[
\nu + \frac{1}{1-\alpha}\sum_{s \in S} p_s(\mathbf y)\,\xi_s
\right].
\label{eq:obj-cvar}
\end{align}

The first-stage terms comprise arc commitments, port--commodity inventory,
and, when enabled, incremental storage construction. The third term is the expected
second-stage recourse cost under the scenario distribution
$p_s(\mathbf y)$. The final term is the CVaR component, written in the
standard Rockafellar--Uryasev auxiliary-variable form \citep{rockafellar2000}.

The parameter $\lambda \in [0,1]$ controls the degree of risk aversion.
When $\lambda = 0$, the model reduces to the risk-neutral formulation.
When $\lambda = 1$, the planner optimises purely with respect to CVaR.
Intermediate values of $\lambda$ generate a tradeoff between average
performance and tail-risk protection.

The parameter $\alpha \in (0,1)$ is the CVaR confidence level. Larger
values of $\alpha$ place greater emphasis on the worst $(1-\alpha)$
fraction of outcomes. For example, $\alpha = 0.9$ means that the model
protects against the average cost incurred in the worst 10\% of the
scenario distribution.

To express CVaR within a tractable optimisation model, we introduce the
auxiliary variable $\nu \ge 0$, which plays the role of the
value-at-risk threshold for the nonnegative recourse cost, together with
nonnegative excess variables
$\xi_s$ for each scenario $s \in S$. These are linked by
\begin{align}
\xi_s &\ge Q_s - \nu, && \forall s \in S,
\label{eq:cvar-excess} \\
\xi_s &\ge 0, && \forall s \in S.
\label{eq:cvar-nonneg}
\end{align}

At optimality, \eqref{eq:cvar-excess}--\eqref{eq:cvar-nonneg} imply
\[
\xi_s = \max\{Q_s-\nu,0\},
\]
so $\xi_s$ measures the amount by which scenario-$s$ recourse cost
exceeds the threshold $\nu$. Consequently, the quantity
\[
\nu + \frac{1}{1-\alpha}\sum_{s \in S} p_s(\mathbf y)\,\xi_s
\]
is exactly the CVaR of the recourse cost distribution at confidence level
$\alpha$. This representation is especially attractive because it
preserves linearity once the decision-dependent probability terms are
handled through the McCormick reformulation described in
\Cref{prop:milp-reform}.

The mean--CVaR formulation therefore yields a planner-facing objective
that balances operational efficiency against resilience to tail events.
For the primary case study, we set $\lambda=0.5$ and select $\alpha=0.80$
using the predeclared tail-composition rule described in
\Cref{sec:cvar-frontier}. The value $\alpha=0.95$ is retained only as a
single-worst-scenario stress setting.

\subsection{Service-Continuity Metrics and Resilience Scope}
\label{sec:service-metrics}

The model evaluates resilience as the ability of a fixed strategic design to
absorb a disruption and adapt through rerouting, supply substitution, and
inventory drawdown. Let $D_c=\sum_{i:(i,c)\in\mathcal D}d_{ic}$. Commodity
service in scenario $s$ is
\begin{align}
R_{cs}
=1-\frac{\sum_{i:(i,c)\in\mathcal D}u_{ic}^{s}}{D_c}.
\label{eq:commodity-service}
\end{align}
We report expected service, worst positive-probability service, and the
probability of falling below threshold $\tau$:
\begin{align}
\bar R_c&=\sum_{s\in S}p_s(\mathbf y)R_{cs},
& R_c^{\min}&=\min_{s:p_s(\mathbf y)>0}R_{cs},
& P_c(\tau)&=\sum_{s\in S}p_s(\mathbf y)\mathbf 1\{R_{cs}<\tau\}.
\label{eq:service-summary}
\end{align}
These commodity-specific measures are absorptive and adaptive service-continuity
measures. The model is
single-period and does not estimate repair duration, infrastructure restoration
paths, or time to recovery.

\section{Tractability and Decomposition Structure}
\label{sec:structural}

We establish two tractability results for the integrated formulation. The
first reformulates the bilinear interaction between decision-dependent
probabilities and recourse as an exact MILP through binary--continuous
McCormick linearisation. The second shows that the mean--CVaR term preserves
scenario-wise separability once the first-stage decisions and VaR threshold are
fixed. These are enabling results rather than standalone theoretical claims;
their role is to certify that the formulation in \Cref{sec:model} can be solved
exactly and independently cross-checked by a standard decomposition.

\subsection{Validated scenario-specific McCormick bounds}
\label{sec:mccormick-bounds}

The McCormick envelopes require a valid upper bound on every feasible value of
the equality-defined recourse variable $Q_s$. An all-shortage policy has cost
$256{,}500$ in the calibrated instance, but that quantity bounds the minimum
recourse \emph{value} from above rather than every feasible value of the
equality-defined variable. It is therefore not used as a McCormick
coefficient.

We compute two valid bounds for each scenario. Under the nonnegative and
finite data assumptions verified by the implementation, an analytical bound is
\begin{align}
M_s^{\mathrm{ana}}
={}&\sum_{a\in A}\sum_{c\in C}
 c_{ac}^{s}\theta_{ac}^{s}\bar u_{ac}^{s}
+\beta_{\rm end}\sum_{(i,c)\in\mathcal K}h_{ic}\bar L_{ic}
+\sum_{(i,c)\in\mathcal D}\pi_c d_{ic},
\label{eq:analytical-bound}
\end{align}
where $\bar L_{ic}=L_{ic}^{0}$ in the base model and equals the largest permitted
post-expansion capacity when expansion is enabled. The bound separately
maximises each nonnegative cost component and ignores coupling, so it is safe
but potentially conservative.

A tighter auxiliary bound is obtained by solving one LP per scenario:
\begin{align}
M_s^{\mathrm{aux}}
=\max\left\{
\sum_{a,c}c_{ac}^{s}x_{ac}^{s}
+\beta_{\rm end}\sum_{(i,c)\in\mathcal K}h_{ic} I_{ic}^{s}
+\sum_{(i,c)\in\mathcal D}\pi_c u_{ic}^{s}:
(x,I,u,g,W)\in\mathcal F_s
\right\},
\label{eq:auxiliary-bound}
\end{align}
where $\mathcal F_s$ is the scenario feasible region with all
candidate arcs available, optional supply lifting, bounded shortage, physical
inventory, and the corridor caps. The adopted coefficient is
\begin{align}
M_s^Q
&=\min\left\{M_s^{\mathrm{ana}},
M_s^{\mathrm{aux}}+10^{-7}\max(1,|M_s^{\mathrm{aux}}|)\right\},
& M_s^{\xi}&=M_s^Q.
\label{eq:adopted-bound}
\end{align}
Because $\nu\ge0$, a canonical optimal excess satisfies
$0\le\xi_s=\max\{Q_s-\nu,0\}\le M_s^Q$. The implementation asserts
$M_s^{\mathrm{aux}}\le M_s^{\mathrm{ana}}$ up to numerical tolerance and
passes the same adopted vector to the extensive form and Benders master.
For $0\le\beta_{\rm end}\le1$, these bounds are valid uniformly; the
primary case uses $\beta_{\rm end}=0$ and the displayed coefficient retains
that parameter explicitly.
Detailed scenario bounds, the full bound comparison, and the cut audit are
reported in Supplementary Material, Sections~S3.3--S3.6 and
Tables~S2--S5.

\begin{proposition}[Exact MILP reformulation under affine decision-dependent probabilities]
\label{prop:milp-reform}
Under the affine probability model
\begin{align}
p_s(\mathbf y)=\bar p_s+\sum_{a\in A}\delta_{sa}y_a,
\qquad s\in S,
\label{eq:ddp-recall}
\end{align}
and valid scenario bounds $0\le Q_s\le M_s^Q$ and
$0\le\xi_s\le M_s^\xi$, the mean--CVaR model admits an exact MILP
reformulation. Introduce
\[
w_{as}=y_aQ_s,\qquad v_{as}=y_a\xi_s,
\qquad a\in A,\ s\in S.
\]
For each $(a,s)$, these products are represented by
\begin{align}
w_{as}&\ge Q_s-M_s^Q(1-y_a), &
v_{as}&\ge \xi_s-M_s^\xi(1-y_a),
\label{eq:mc-lb}\\
w_{as}&\le Q_s, & v_{as}&\le\xi_s,
\label{eq:mc-ub1}\\
w_{as}&\le M_s^Qy_a, & v_{as}&\le M_s^\xi y_a,
\label{eq:mc-ub2}\\
w_{as}&\ge0, & v_{as}&\ge0.
\label{eq:mc-nn}
\end{align}
The bound is scenario-specific and does not depend on $a$; the same $M_s^Q$
is reused for all arc products in scenario $s$.
\end{proposition}

\begin{proof}
Substituting \eqref{eq:ddp-recall} into \eqref{eq:obj-cvar} produces
binary--continuous products $y_aQ_s$ and $y_a\xi_s$. For $y_a=0$,
\eqref{eq:mc-lb}--\eqref{eq:mc-nn} force $w_{as}=v_{as}=0$; for $y_a=1$,
they force $w_{as}=Q_s$ and $v_{as}=\xi_s$. Hence the envelopes are exact for
binary $y_a$. Replacing the products by $w_{as}$ and $v_{as}$ yields an
equivalent MILP. Validity depends on the bounds covering the feasible range of
the equality-defined continuous variables, which is ensured by
\eqref{eq:analytical-bound}--\eqref{eq:adopted-bound}.
\end{proof}

\begin{proposition}[Scenario separability under mean--CVaR]
\label{prop:cvar-decomp}
For fixed first-stage decisions $(\mathbf{y}, \mathbf{W}, \nu)$, the
second-stage problem remains separable across scenarios. In particular,
for each scenario $s \in S$, the associated recourse variables and CVaR
excess variable $\xi_s$ can be optimised independently of all other
scenarios.
\end{proposition}

\begin{proof}
Fix $(\mathbf{y}, \mathbf{W}, \nu)$. Then the decision-dependent
probabilities $p_s(\mathbf{y})$ are constants, and the first-stage
variables no longer appear as optimisation variables in the recourse
stage.

The uncertain part of the objective is
\[
(1-\lambda)\sum_{s\in S} p_s(\mathbf{y}) Q_s
\;+\;
\frac{\lambda}{1-\alpha}\sum_{s\in S} p_s(\mathbf{y}) \xi_s,
\]
since the term $\lambda \nu$ is constant once $\nu$ is fixed. For each
scenario $s$, the CVaR constraints are
\[
\xi_s \ge Q_s - \nu, \qquad \xi_s \ge 0.
\]
Hence, at optimality,
\[
\xi_s = \max\{Q_s-\nu,0\},
\]
so the contribution of scenario $s$ to the objective is
\[
(1-\lambda)p_s(\mathbf{y})Q_s
\;+\;
\frac{\lambda}{1-\alpha}p_s(\mathbf{y})\max\{Q_s-\nu,0\}.
\]

Moreover, the recourse constraints for scenario $s$ (flow balance,
capacity, admissibility, storage, and corridor-cap constraints) involve
only scenario-$s$ variables. No recourse variable from scenario $s$ appears
in the constraints of any other scenario $s' \neq s$.

Therefore, once $(\mathbf{y}, \mathbf{W}, \nu)$ is fixed, the second-stage
optimisation decomposes into independent scenario subproblems, one for
each $s \in S$. The only cross-scenario coupling is through the shared
first-stage decisions, which are handled in the master problem.
\end{proof}

\begin{remark}[Standard decomposition convergence]
\label{rem:benders-finite}
For fixed first-stage decisions, each scenario subproblem is a
feasible and bounded linear program. The usual finite-convergence argument for
multi-cut Benders decomposition therefore applies \citep{birge1997introduction}:
each dual extreme point
defines a valid recourse under-estimator, and only finitely many distinct
extreme-point cuts can be generated. We use this standard result only to
justify the verification implementation; no new convergence theorem is
claimed.
\end{remark}

\section{Exact Solution and Computational Verification}
\label{sec:algorithm}

The extensive-form MILP is the primary solution method for all reported policy
experiments. To independently verify the implementation, we also constructed a
standard scenario-wise multi-cut Benders decomposition for the fixed-capacity
formulation. Its master contains the first-stage decisions and bounded
scenario recourse surrogates, while each scenario subproblem solves
the physical multi-commodity recourse LP. The identical scenario-specific
McCormick-bound vector is used in both formulations.

The verification is deliberately limited in scope. An optional corridor-group
cut family is disabled because its global validity was not established
independently of the generating design. Learning-augmented warm starts, learned
cut pruning, and SAA components are outside the reported analysis. Across the
four principal
fixed-capacity small-support cases, the maximum extensive-form/Benders objective
difference is $4.37\times10^{-11}$ and the maximum scenario-cost difference is
$7.28\times10^{-12}$. These are not global maxima: in the 729-scenario run the
corresponding differences are $2.054947\times10^{-4}$ and 137.5, respectively,
with the same first-stage policy. In the 729-scenario verification,
the extensive form is faster than the Benders implementation. Runtime is
reported only for the 729-scenario verification; we therefore make no
cross-support runtime comparison. The complete master/subproblem formulation,
pseudocode, cut-validity audit, and computational comparison are provided in
Supplementary Material, Sections~S2--S3, Algorithm~S1, and Table~S5.

\section{Case Study}
\label{sec:computational}

The case study uses a 16-node, 28-arc data-anchored, stylized India-inspired
network grounded in Indian maritime energy-import context, four commodities (crude oil, LNG, LPG, and fertilizer), and a
nine-scenario disruption library. The computational implementation is written in
Python and solved with Gurobi~13.0.2 using one thread, seed 0, presolve level 2,
cuts level 2, and MIPGap $10^{-9}$. Unless otherwise stated, the primary
risk-averse case uses $\lambda=0.5$ and $\alpha=0.80$. The confidence level is
not chosen from a design optimised at that same $\alpha$; it is selected by a
predeclared complete-mass rule applied to the DDU risk-neutral
reference design (\Cref{sec:cvar-frontier}).

\subsection{Canonical DDU Mean--CVaR Model}
\label{sec:base-case}

The base case includes decision-dependent probabilities, strategic
inventory, corridor concentration limits, and mean--CVaR risk aversion. The
planner activates 21 of 28 candidate arcs and pre-positions 4,250 inventory
units. The design uses Hormuz, Cape, direct, and pipeline-bypass alternatives,
rather than relying on one least-cost corridor.

\begin{figure}[t]
\centering
\includegraphics[width=\textwidth,trim=68 28 73 41,clip]{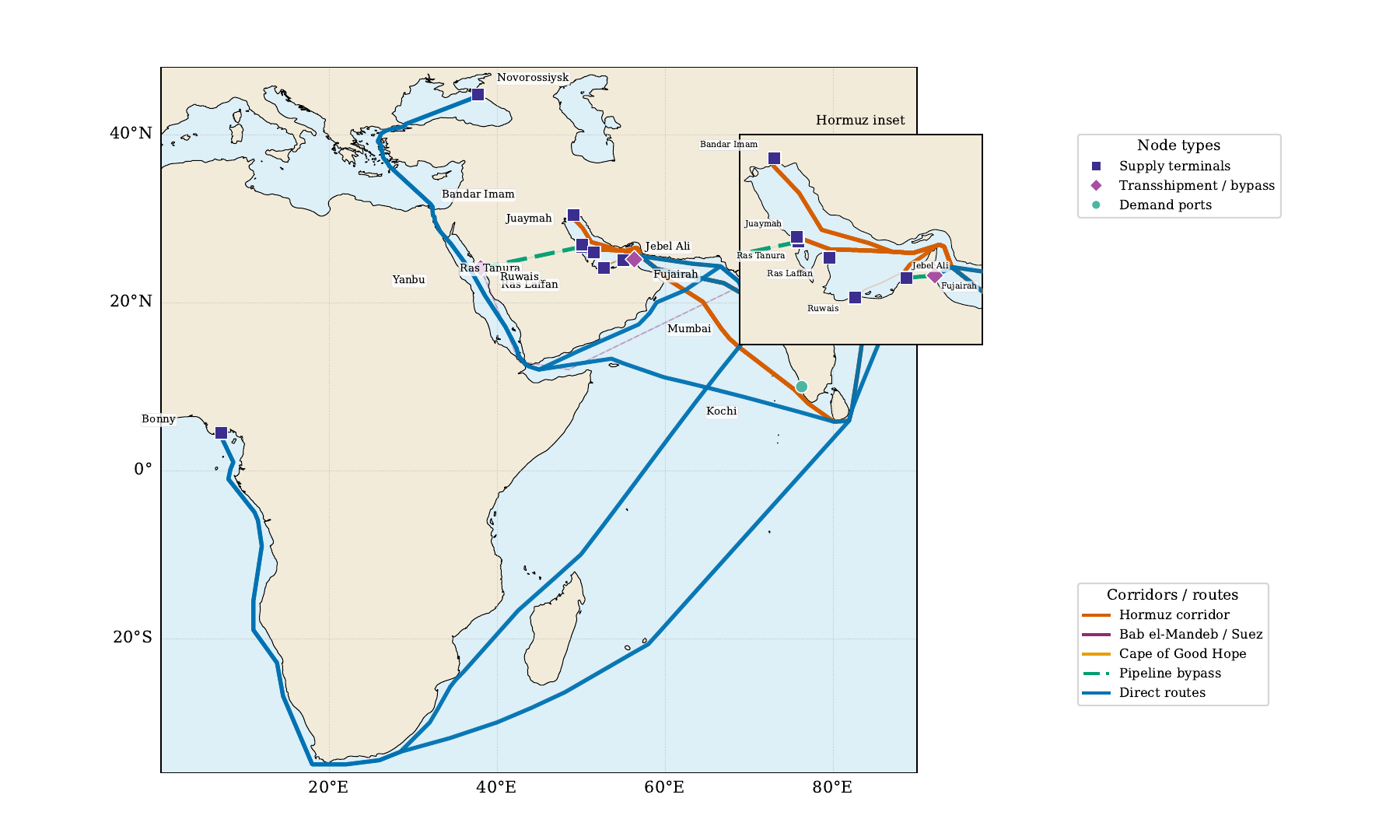}
\caption{Base-case first-stage network design. Highlighted arcs are
activated before the disruption scenario is realised.}
\label{fig:base-network}
\end{figure}

\begin{table}[t]
\centering
\caption{Canonical DDU mean--CVaR base-case results
($\lambda=0.5$, $\alpha=0.80$, $\beta_{\rm end}=0$).}
\label{tab:base-headline}
\scriptsize
\begin{tabular}{rrrrrrr}
\toprule
Objective & First-stage cost & Expected recourse & CVaR$_{0.80}$ & VaR & Active arcs & Total $W$ \\
\midrule
73,849.884188 & 1,552.00 & 56,672.690875 & 87,923.0775 & 80,394.375 & 21 & 4,250 \\
\bottomrule
\end{tabular}
\end{table}

The scenario recourse costs range from 23,482.50 in the normal state to
89,400.00 in the Hormuz-closure and closure-with-bypass states. The selected
VaR threshold is 80,394.375 and the CVaR tail contains the dual-disruption
state together with the two 89,400.00 closure states. The canonical extensive
form and Benders decomposition differ by at most $8.73\times10^{-11}$ in
objective value.

\begin{figure}[H]
\centering
\begin{subfigure}[t]{\textwidth}
\centering
\includegraphics[width=\linewidth]{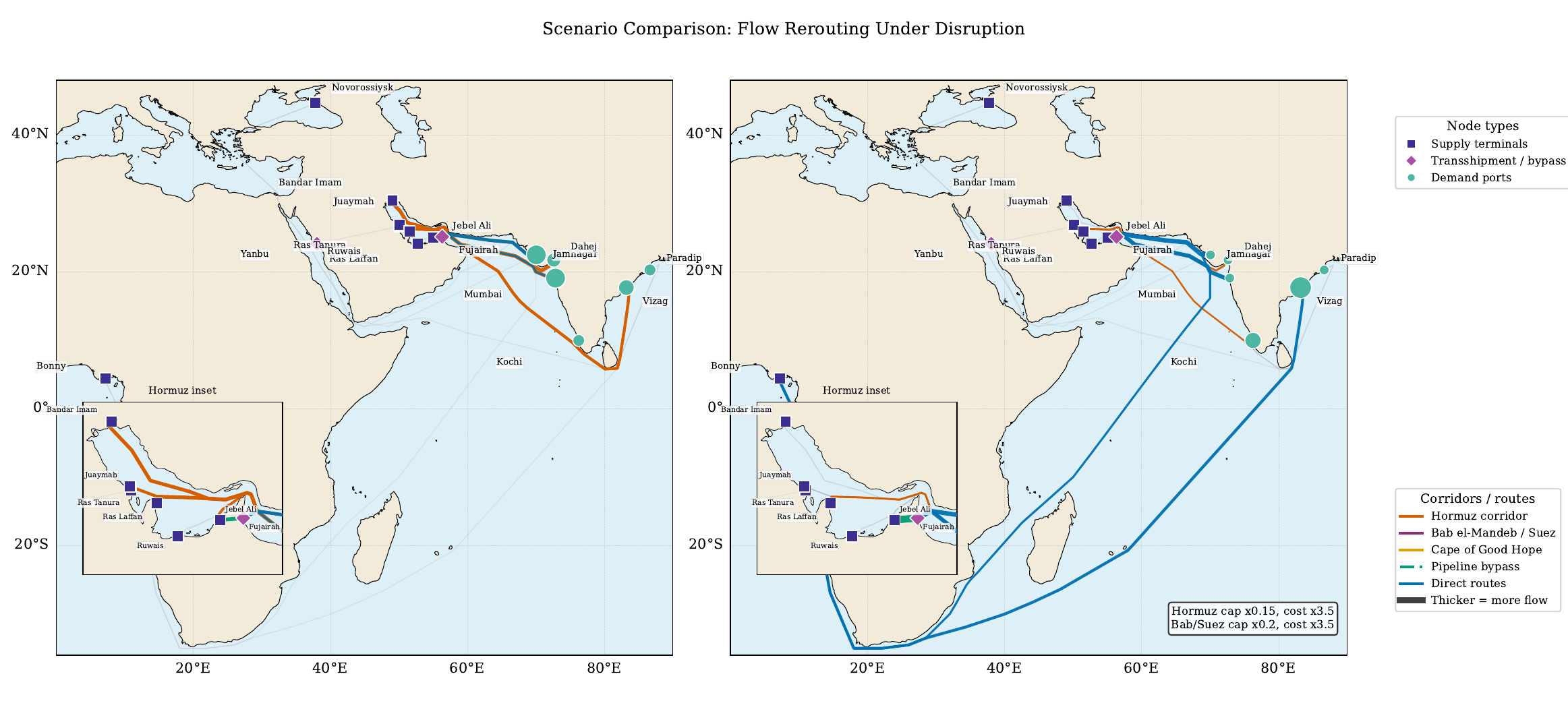}
\caption{Scenario-contingent rerouting.}
\end{subfigure}
\par\medskip
\begin{subfigure}[t]{0.78\textwidth}
\centering
\includegraphics[width=\linewidth]{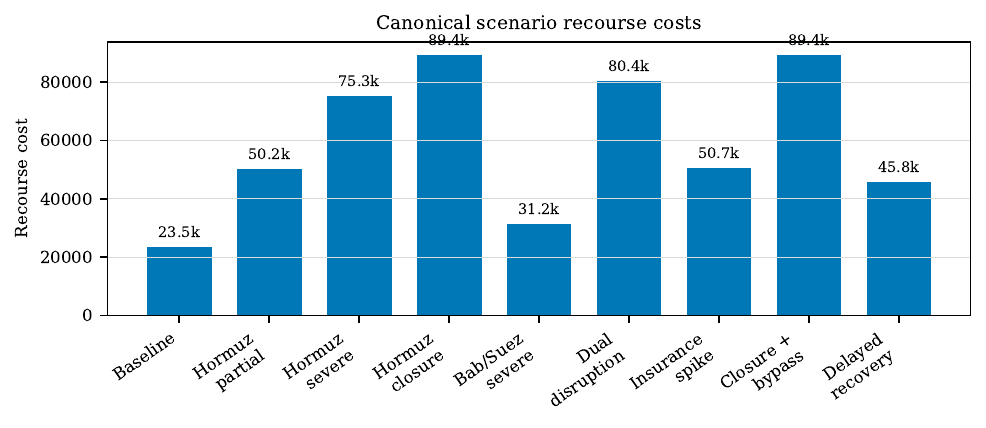}
\caption{Scenario recourse costs.}
\end{subfigure}
\caption{Operational response under the fixed first-stage design.}
\label{fig:base-recourse}
\end{figure}

The transport solution was audited independently of the optimiser's objective.
Across all nine scenarios and four commodities, the maximum mass-balance
residual is $6.82\times10^{-13}$ and the maximum violation of activation,
capacity, admissibility, storage, supply, shortage, or corridor restrictions is
$1.82\times10^{-12}$. All positive flows decompose into 81
source-to-destination commodity paths with zero undecomposed remainder. The
full scenario audit is reported in \ref{app:operational-audits}.

The optimised inventory fills all available crude, LNG, and LPG storage and
uses 700 of 800 normalised fertilizer-storage units. Crude oil and fertilizer
have no shortage in any calibrated scenario. Commodity exposure is less
uniform: LNG and LPG reach worst positive-probability service of 0.346154 and
0.333333, respectively.

\begin{table}[t]
\centering
\caption{Commodity service continuity under the canonical
base-case design. Probabilities use the DDU-adjusted distribution.}
\label{tab:base-service}
\scriptsize
\begin{tabular}{lrrrrrr}
\toprule
Measure & Expected service & Worst service & $P(R<0.80)$ & $P(R<0.90)$ &
$P(R<0.95)$ & Expected unmet \\
\midrule
Crude & 1.000 & 1.000 & 0.000 & 0.000 & 0.000 & 0.00 \\
LNG & 0.838745 & 0.346154 & 0.1672 & 0.3949 & 0.5954 & 209.632 \\
LPG & 0.684546 & 0.333333 & 0.5954 & 0.6454 & 0.6454 & 378.545 \\
Fertilizer & 1.000 & 1.000 & 0.000 & 0.000 & 0.000 & 0.00 \\
\bottomrule
\end{tabular}
\end{table}

The formulation permits port--commodity holding costs $h_{ic}$. Because
consistent empirical site-level estimates are unavailable, the calibrated
baseline sets all port multipliers to one and treats heterogeneity as
sensitivity. Across 25 regional and one-site pre-positioning-cost perturbations
with $\beta_{\rm end}=0$, the objective changes by at most 480 (0.649967\%),
while the arc vector and inventory allocation are unchanged. This is evidence
of stability over the tested cost range, not a
claim that spatial holding-cost variation is universally irrelevant.

The DDU-adjusted probabilities differ only modestly from the baseline
distribution, but the shifts are not zero. For the base-case design, the
largest absolute shift is 0.0046. Full-precision baseline and adjusted
probabilities are reported in \ref{app:basecase-additional}.

\subsection{Value of Stochastic Modelling}
\label{sec:value-stochastic}

We compare the risk-neutral stochastic program (SP) with the expected-value
(EV) approximation and the wait-and-see (WS) perfect-information benchmark.
The EV first-stage route and inventory decisions are fixed and evaluated over
the full distribution to obtain EEV. Thus,
\[
\mathrm{VSS}=\mathrm{EEV}-\mathrm{SP},\qquad
\mathrm{EVPI}=\mathrm{SP}-\mathrm{WS}.
\]

For the calibrated nine-scenario library, the SP objective is
58,402.218309, whereas the EV design costs 59,189.591250 when evaluated under the full
distribution. The resulting VSS is 787.372941, or 1.348190\% of the SP objective.
The WS value is 58,386.298309, giving an EVPI of 15.920000 (0.027259\%). The expected
ordering $\mathrm{WS}\leq\mathrm{SP}\leq\mathrm{EEV}$ holds. Hence the main
gain comes from selecting a first-stage policy against the complete scenario
distribution, while perfect ex-post information adds little once that policy
has been chosen.

Support refinements confirm that this conclusion is not
peculiar to the nine-point prior. Table~\ref{tab:vss-supports} reports the same
benchmarks on 27- and 81-scenario constructions. VSS is 804.852941
(1.574932\%) in both supports, whereas EVPI remains below 0.03\%.
The dedicated 135-scenario selective-admissibility experiment is a separate
variant, not a generic support-size VSS calculation. These are value-of-stochastic-solution and
perfect-information comparisons; they do not establish superiority over
robust or distributionally robust alternatives.

\begin{table}[t]
\centering
\caption{Value of stochastic modelling and perfect information across scenario supports. All SP, EV, EEV, and WS solves are optimal.}
\label{tab:vss-supports}
\small
\begin{tabular}{rrrrrrr}
\toprule
$|S|$ & SP & EEV & VSS & VSS (\%) & WS & EVPI (\%) \\
\midrule
9   & 58,402.218309 & 59,189.591250 & 787.372941 & 1.348 & 58,386.298309 & 0.027 \\
27  & 51,103.978134 & 51,908.831076 & 804.852941 & 1.575 & 51,093.226512 & 0.021 \\
81  & 51,103.978134 & 51,908.831076 & 804.852941 & 1.575 & 51,093.226512 & 0.021 \\
\bottomrule
\end{tabular}
\end{table}

\noindent\emph{Note.} The identical 27- and 81-scenario values are verified
outputs rather than duplicated entries. The 27-scenario construction fixes the
pipeline factor at one level, whereas the 81-scenario construction crosses that
factor; the support relationship is described in
\ref{app:dependence-details}. The equality of these aggregate metrics in the
reported solves should not be interpreted as evidence that pipeline variation
is generally irrelevant.

\begin{figure}[t]
\centering
\includegraphics[width=0.70\textwidth]{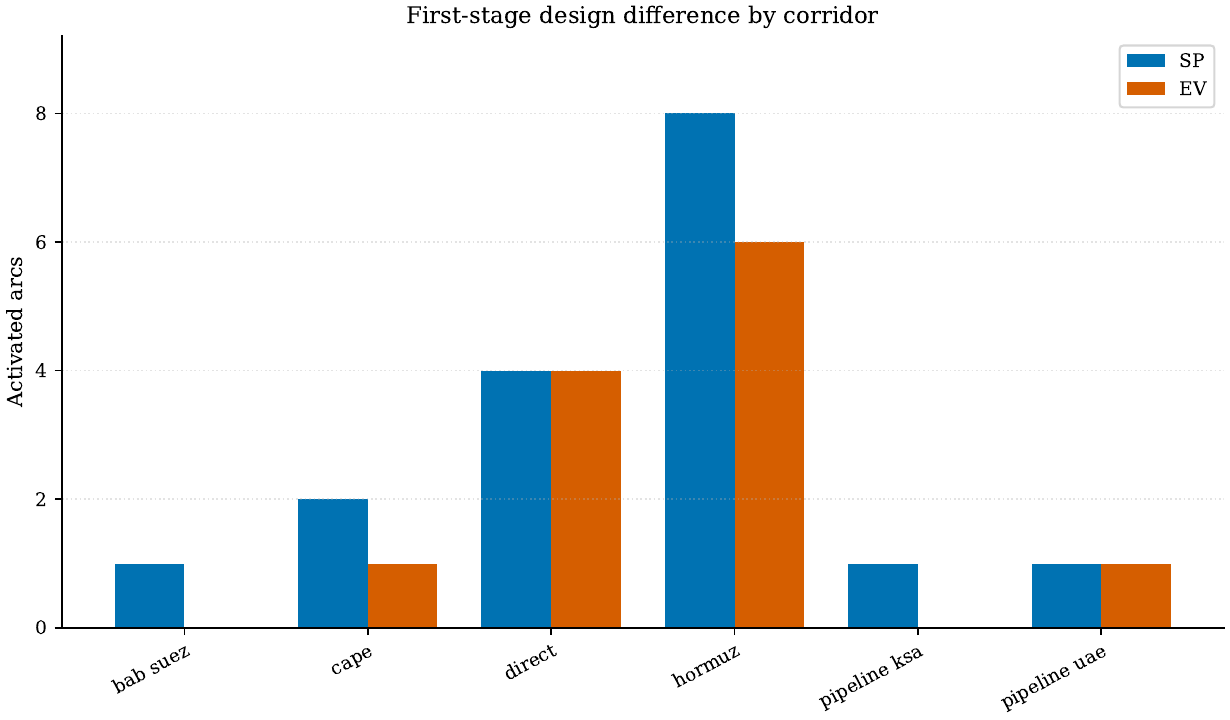}
\caption{First-stage corridor activations under the calibrated nine-scenario stochastic-program and expected-value designs. The canonical SP design activates 18 arcs, compared with 13 for EV.}
\label{fig:sp-ev-design}
\end{figure}

\subsection{Dependence Diagnostics and Fixed-Marginal Stress Tests}
\label{sec:dependence-diagnostics}

Let $H$ and $B$ denote binary Hormuz and Bab/Suez disruption indicators.
For the calibrated library,
\[
P(H=1)=0.75,\qquad P(B=1)=0.40,\qquad P(H=1,B=1)=0.30.
\]
The binary incidence table therefore equals the product of its marginals to
machine precision, with phi correlation $-2.62\times10^{-16}$. This does not
remove jointly severe states. The encoded capacity-loss severities have
unconditional correlation $-0.2860$ and correlation $-0.0804$ within jointly
disrupted states. These negative values are properties of the
author-specified scenario weights, not empirical estimates of physical
negative dependence: the library assigns comparatively more mass to narratives
in which one corridor is severely disrupted than to simultaneous
maximum-severity states.
We therefore do not interpret the severity coefficients causally. A
support-restricted, marginal-preserving I-projection changes the probability
vector by KL divergence $0.0063$ nats and total-variation distance $0.0432$,
but does not change the optimal design ($\VMDIPF=0$).

To test incidence dependence directly, we use the smallest non-degenerate
fully crossed support, with 81 scenarios. We hold
$P(H=1)=0.75$ and $P(B=1)=0.40$ fixed and vary
$q=P(H=1,B=1)$ over 21 points in the Fr\'echet-feasible interval
$[0.15,0.40]$. Conditional severity profiles and the Cape and pipeline
marginals are preserved to machine precision. The primary factorial library
isolates incidence, severity, Cape, and pipeline channels; a separate
135-scenario library restores selective commodity admissibility.

\begin{table}[t]
\centering
\caption{Fixed-marginal dependence stress test. The final column is
the maximum regret from using the independence-optimised design; it is
numerical zero in every primary cell.}
\label{tab:dependence-grid}
\scriptsize
\begin{tabular}{clrrrrr}
\toprule
Config. & Network modification & $Z(q_{\min})$ & $Z(q_{\max})$ &
Range & $dZ/dq$ & Max regret \\
\midrule
r0 & Full calibrated network & 48,669.44 & 48,694.44 & 25.00 & 100.00 & 0 \\
r1 & Pipeline bypass capacity halved & 51,029.20 & 51,374.53 & 345.32 & 1,381.29 & 0 \\
r2 & Cape capacity halved & 50,606.83 & 50,631.83 & 25.00 & 100.00 & 0 \\
r3 & Cape and direct-route capacity halved & 52,686.09 & 52,717.34 & 31.25 & 125.00 & 0 \\
r4 & Pipeline, Cape, and direct-route capacity halved & 54,759.81 & 55,384.62 & 624.82 & 2,499.27 & 0 \\
r5 & r4 plus storage capacity halved & 88,027.43 & 89,562.27 & 1,534.84 & 6,139.36 & 0 \\
r6 & Storage capacity halved only & 76,049.95 & 76,969.98 & 920.03 & 3,680.13 & 0 \\
r7 & Direct-route capacity halved only & 50,748.70 & 50,779.95 & 31.25 & 125.00 & 0 \\
\bottomrule
\end{tabular}
\end{table}

Across all $21\times8=168$ primary cells, the design optimised at the
independence point $q=0.30$ remains optimal. No arc or inventory decision
changes, and the full first-stage solution is numerically unique within the
stated tolerances in every cell. The result is not that dependence has no
economic effect. In the full network, the objective rises by 25.00 units over
the feasible $q$ range, corresponding to $dZ/dq=100$. In the 81-scenario factorial family, the optimal first-stage design activates
the single Bab/Suez arc. Nevertheless, after averaging over the preserved
severity, Cape, and pipeline factors, the Bab-only incidence cell has almost
the same recourse cost as the no-disruption cell. The positive slope with
respect to $q$ therefore arises primarily from the joint-disruption interaction
rather than from a standalone Bab-only penalty. Dependence changes expected
cost, but the calibrated first-stage policy is sufficiently stable that
reoptimisation does not improve it in the 168 primary cells.

\begin{figure}[t]
\centering
\includegraphics[width=0.98\textwidth]{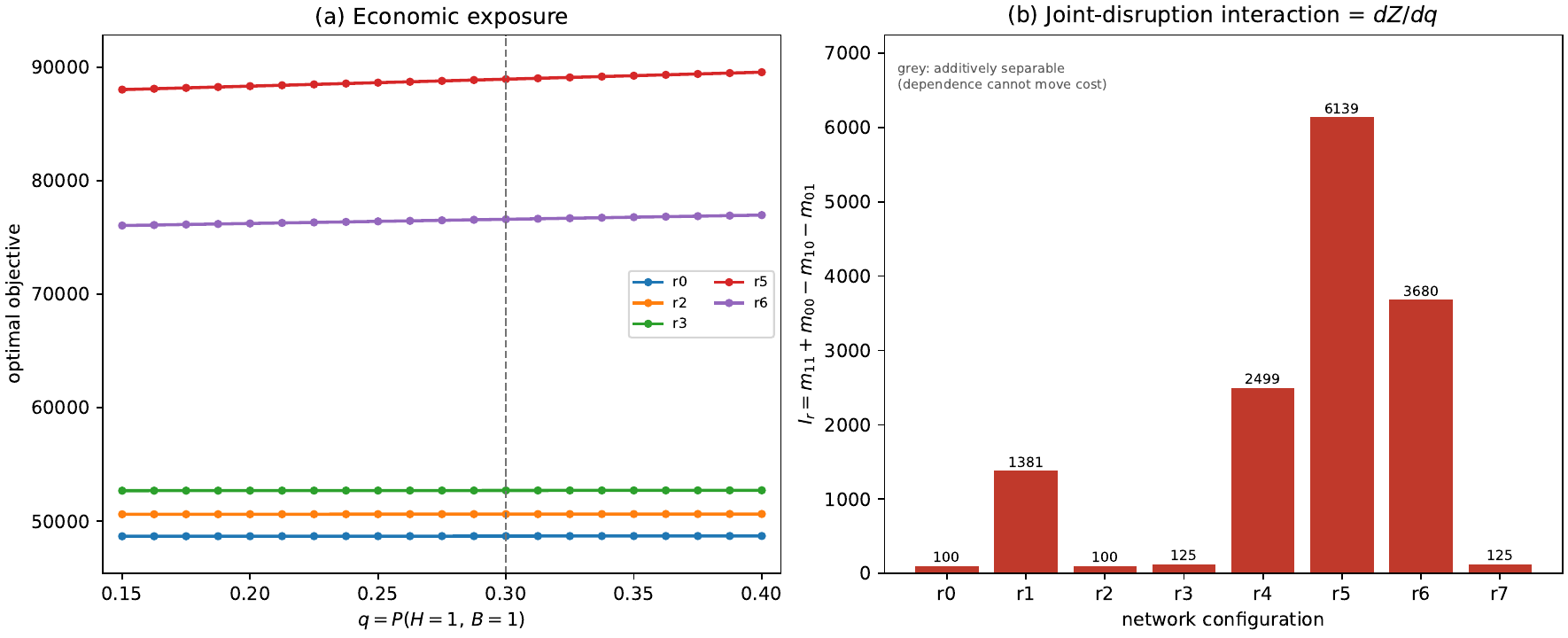}
\caption{Economic exposure and the joint-disruption interaction across the
eight network configurations. The objective changes with $q$, but the
independence-optimised design remains optimal in all 168 primary cells.}
\label{fig:dependence-interaction}
\end{figure}

The exact parameter-sensitivity programme contains 96 attempted cells: 90
optimal cells, six infeasible cells, and six significant cells under the
achieved-gap significance rule. The maximum feasible regret is $0.002213773\%$.
Scaling activation and holding costs over 144 cells produces 12 confirmed
policy-changing cells; the largest regret is $1.1673\%$. These positive cases
demonstrate that the zero regrets in the primary fixed-marginal grid do not
arise from an inert evaluation metric. The dedicated 135-scenario
selective-admissibility experiment has 66 reported cells, no significant
regrets, zero policy distance, and worst-case LPG service from 0.166667 to
0.333333.

\subsection{Value of Decision-Dependent Probabilities}
\label{sec:vep}

To isolate the value of probability endogeneity, the VEP experiment uses the
risk-neutral, no-expansion configuration, with all other model inputs held
fixed.
At each exposure-sensitivity multiplier $k$, we solve the DDU-aware model,
solve the fixed-probability counterfactual, and evaluate the counterfactual
first-stage design under the DDU distribution. We define
\[
\mathrm{VEP}(k)=
Z_k^{\mathrm{DDU}}\!\left(\mathbf y_k^{\mathrm{NoDDU}},
\mathbf W_k^{\mathrm{NoDDU}}\right)
-
Z_k^{\mathrm{DDU}}\!\left(\mathbf y_k^{\mathrm{DDU}},
\mathbf W_k^{\mathrm{DDU}}\right).
\]
The complete global-validity audit gives
$k_{\max}=6.71296$; the reported boundary point is $k=6.7$.

\begin{table}[t]
\centering
\caption{VEP sweep. Probability shifts are retained at full
precision and the probability-validity constraints hold globally at every
reported multiplier.}
\label{tab:vep}
\scriptsize
\begin{tabular}{rrrrrrrr}
\toprule
$k$ & DDU objective & No-DDU under DDU & VEP & VEP (\%) &
DDU arcs & No-DDU arcs & Max $|\Delta p_s|$ \\
\midrule
0 & 58,402.22 & 58,402.22 & 0.00 & 0.00\% & 18 & 18 & 0.0000 \\
0.5 & 58,339.14 & 58,508.15 & 169.01 & 0.29\% & 23 & 18 & 0.0000 \\
1 & 58,224.69 & 58,614.09 & 389.40 & 0.67\% & 21 & 18 & 0.0000 \\
2 & 57,920.07 & 58,825.95 & 905.89 & 1.56\% & 21 & 18 & 0.0100 \\
3 & 57,591.27 & 59,037.82 & 1,446.55 & 2.51\% & 20 & 18 & 0.0200 \\
5 & 56,837.23 & 59,461.56 & 2,624.33 & 4.62\% & 20 & 18 & 0.0300 \\
6.7 & 56,196.29 & 59,821.73 & 3,625.44 & 6.45\% & 20 & 18 & 0.0400 \\
\bottomrule
\end{tabular}
\end{table}

At the analyst-set scale $k=1$, VEP is 389.40, or 0.67\% of the DDU
objective. It rises to 2,624.33 (4.62\%) at $k=5$ and 3,625.44 (6.45\%) at
$k=6.7$. VEP is the model-conditional value of accounting for the
analyst-specified DDU probability mechanism, not a general economic value of
endogenous real-world disruption risk. The route portfolio,
rather than total storage, carries the adjustment: total inventory remains
4,250 normalised inventory units throughout, while the DDU-aware design uses
21 arcs at $k=1$ compared with 18 in the no-DDU design. The probability mass residual is below
$10^{-18}$ and the smallest globally feasible probability remains positive
through the reported range.

\begin{figure}[t]
\centering
\includegraphics[width=\textwidth]{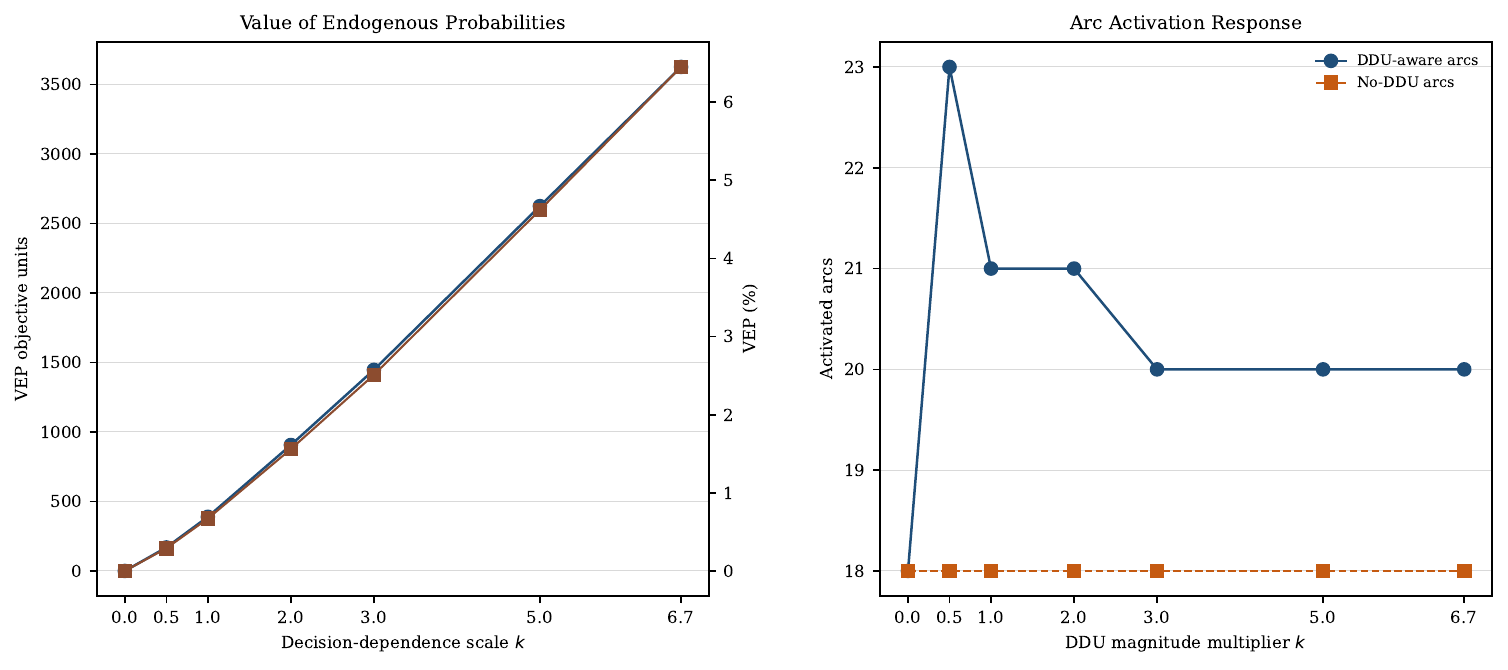}
\caption{Economic value of endogenous probabilities over the globally valid
exposure-sensitivity range.}
\label{fig:vep}
\end{figure}

\begin{figure}[t]
\centering
\includegraphics[width=\textwidth]{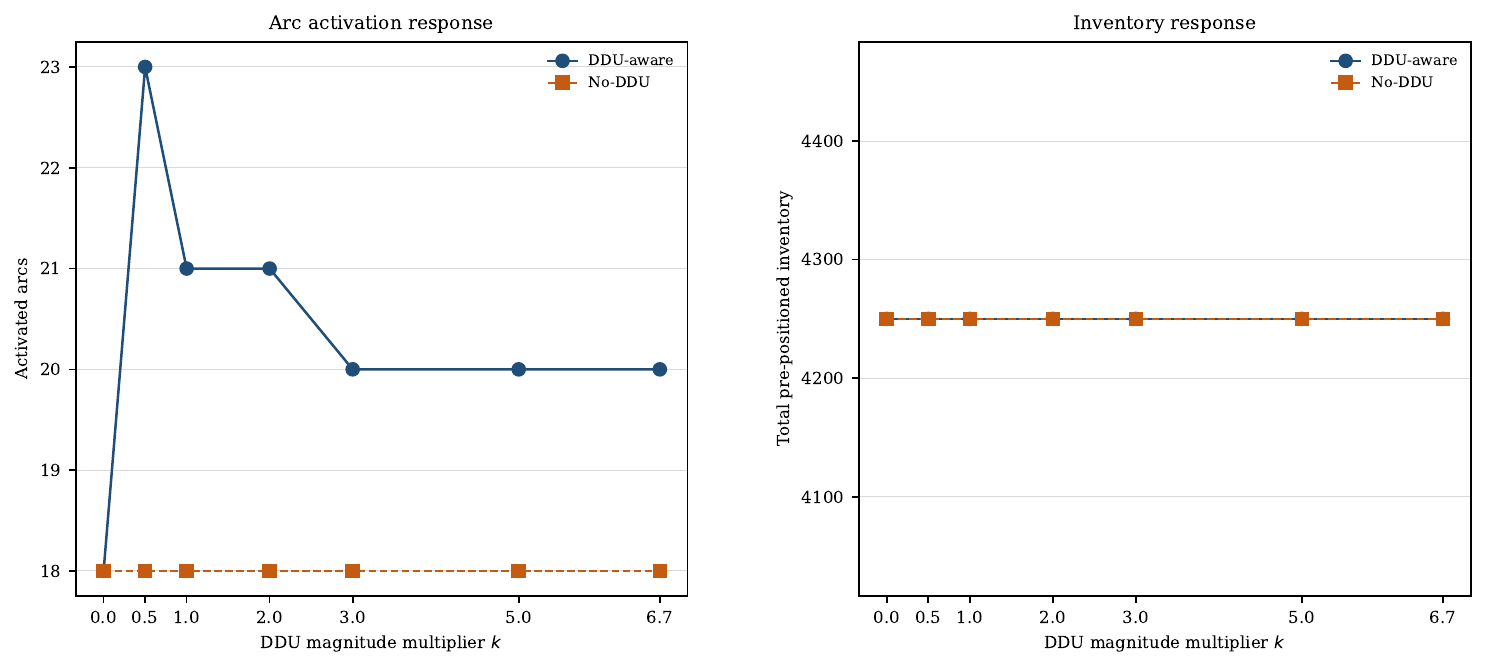}
\caption{First-stage route-design response to the exposure-sensitivity
multiplier.}
\label{fig:vep-design-response}
\end{figure}

\subsection{Risk Aversion, Tail Calibration, and Storage Expansion}
\label{sec:cvar-frontier}

The primary confidence level is selected before the risk-averse model is
optimised. We rank scenario costs under the DDU risk-neutral design
and apply predeclared rule B: retain a tail containing at least two complete
scenario masses. This selects $\alpha=0.80$. The corresponding tail contains
two complete scenario masses and two distinct loss levels. At $\alpha=0.85$,
the two 89,400.00 closure states alone exceed the 15\% tail mass, so
VaR=CVaR=89,400 and the tail collapses to one loss level; it is retained only
as a sensitivity setting.
At $\alpha=0.95$, the tail mass is smaller than the probability of the worst
scenario; CVaR therefore collapses to that single scenario and is retained
only as a stress setting.

\begin{table}[t]
\centering
\caption{$\alpha$ sweep at $\lambda=0.5$.}
\label{tab:alpha-sweep}
\small
\begin{tabular}{rrrrrrr}
\toprule
$\alpha$ & Objective & Expected recourse & CVaR & VaR & Arcs &
Complete scenario masses \\
\midrule
0.50 & 67,702.74 & 56,672.69 & 75,628.79 & 50,216.75 & 21 & 5 \\
0.60 & 70,827.62 & 56,672.69 & 81,878.55 & 50,651.91 & 21 & 4 \\
0.70 & 72,126.91 & 56,672.69 & 84,477.14 & 75,342.00 & 21 & 3 \\
0.80 & 73,849.88 & 56,672.69 & 87,923.08 & 80,394.38 & 21 & 2 \\
0.85 & 74,588.35 & 56,672.69 & 89,400.00 & 89,400.00 & 21 & 0 \\
0.90 & 74,588.35 & 56,672.69 & 89,400.00 & 89,400.00 & 21 & 0 \\
0.95 & 74,588.35 & 56,672.69 & 89,400.00 & 89,400.00 & 21 & 0 \\
0.99 & 74,588.35 & 56,672.69 & 89,400.00 & 89,400.00 & 21 & 0 \\
\bottomrule
\end{tabular}
\end{table}

For the reported risk-aversion analysis, $\lambda=0.5$ is fixed and only
$\alpha$ is varied; no joint $(\lambda,\alpha)$ grid is reported.

\begin{table}[t]
\centering
\caption{Primary mean--CVaR point at $\alpha=0.80$, $\lambda=0.5$.}
\label{tab:lambda-frontier}
\small
\begin{tabular}{rrrrr}
\toprule
$\lambda$ & Objective & Expected recourse & CVaR & Active arcs \\
\midrule
0.5 & 73,849.88 & 56,672.69 & 87,923.08 & 21 \\
\bottomrule
\end{tabular}
\end{table}

\begin{figure}[t]
\centering
\includegraphics[width=\textwidth]{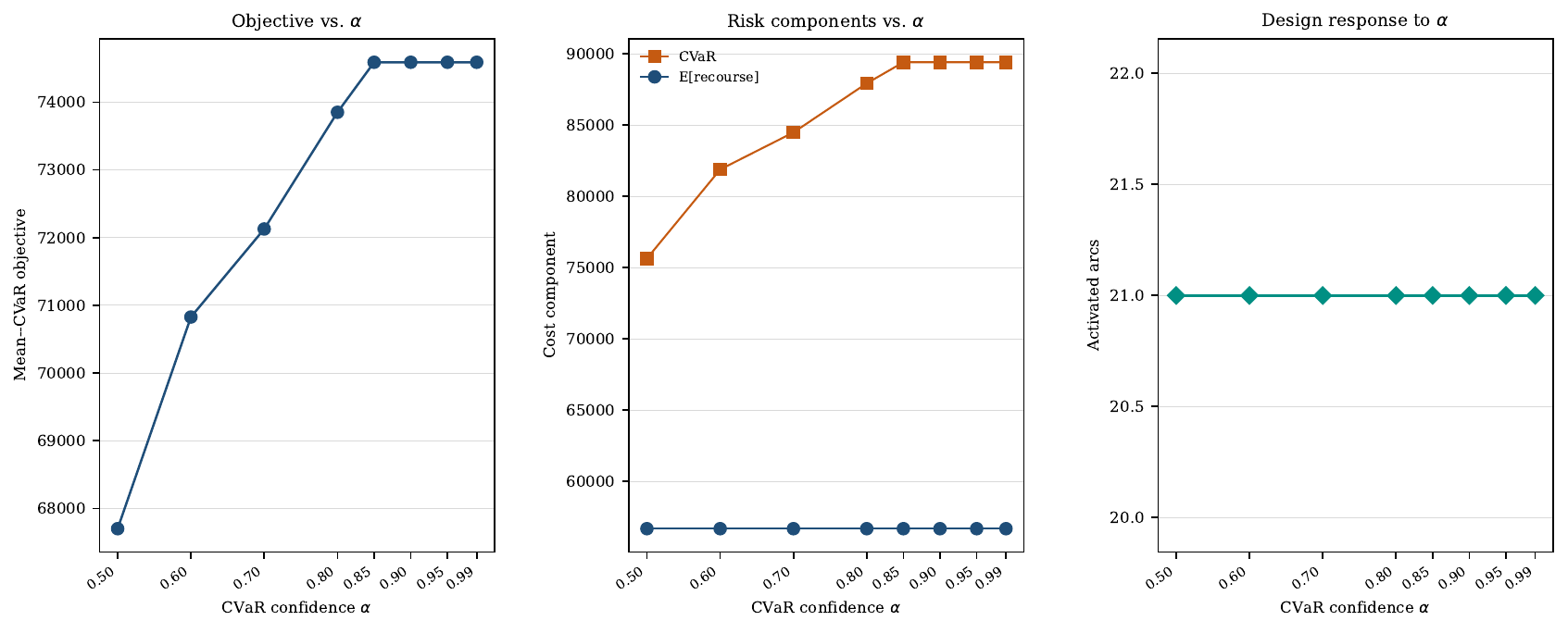}
\caption{Tail-risk response across the CVaR confidence-level sweep.}
\label{fig:cvar-sensitivity}
\end{figure}

Because existing storage is nearly fully used, we also test explicit
capacity expansion. The no-expansion model remains the canonical benchmark
because it represents the current installed network and provides a common
basis for the VSS, VEP, dependence, and risk-aversion experiments. The capex
parameters are assumed policy-stress values rather than empirically calibrated
project costs; expansion results are therefore reported separately and are not
presented as an empirical investment prescription.

\begin{table}[t]
\centering
\caption{Storage-expansion stress test with explicit construction-cost
decomposition. Existing installed capacity is sunk; only incremental capacity
is charged. Added capacity is reported in normalised model units.}
\label{tab:storage-expansion}
\scriptsize
\begin{adjustbox}{max width=\textwidth}
\begin{tabular}{lrrrrrrrrr}
\toprule
Case & $\kappa$ & Objective & Fixed constr. & Variable constr. & Total constr. &
Constr. share (\%) & Normalised capacity added & Pairs & Arcs \\
\midrule
none (base) & -- & 73,849.884 & 0 & 0 & 0 & 0.00 & 0 & 0 & 21 \\
low & 2 & 52,675.122 & 1,370 & 1,370 & 2,740 & 5.20 & 1,775 & 7 & 20 \\
medium & 10 & 58,155.122 & 6,850 & 1,370 & 8,220 & 14.14 & 1,775 & 7 & 20 \\
high & 50 & 70,935.430 & 5,000 & 200 & 5,200 & 7.33 & 200 & 2 & 21 \\
\bottomrule
\end{tabular}
\end{adjustbox}
\end{table}

The medium case activates expansion at seven predefined port--commodity pairs,
adds 1,775 normalised capacity units, and reduces the
objective by 21.24\% relative to the no-expansion base. Its construction cost
is 8,220, equal to 14.14\% of total objective. The number of active corridor
commitments falls from 21 to 20,
showing that strategic storage can substitute for some ex ante transport
capacity. Under the assumed high-capex specification, the only retained
expansion is 200 normalised capacity units across two pairs; this is a
regime-specific result rather than an empirically calibrated investment
ranking.

\section{Computational Verification and Sensitivity Study}
\label{sec:computational-study}

\subsection{McCormick-Bound Validation and Exact Formulation Check}
\label{sec:frz4-computational}

We evaluate the adopted scenario-specific bounds using one thread, seed 0,
presolve 2, cuts 2, and MIPGap $10^{-9}$. The canonical nine-scenario
extensive-form/Benders objective difference is $8.73\times10^{-11}$.
The dedicated 729-scenario verification is reported separately.

\begin{table}[t]
\centering
\caption{Extensive-form/Benders verification using the adopted valid bounds. The extensive form is faster in the verified 729-scenario run.}
\label{tab:frz4-main}
\small
\begin{tabular}{rrrrrr}
\toprule
$|S|$ & EF (s) & BD (s) & BD iterations & BD cuts & Status \\
\midrule
9   & -- & -- & -- & -- & objective agreement $8.73\times10^{-11}$ \\
729 & 1.209 & 8.550 & 12 & 7,272 & OPT/OPT \\
\bottomrule
\end{tabular}
\vspace{2pt}
\begin{minipage}{0.97\linewidth}
\footnotesize
\textit{Note:} The 729-scenario row is a bounded 729-scenario verification,
not a claim of computational superiority for Benders.
\end{minipage}
\end{table}

The canonical formulations recover equivalent first-stage designs and
scenario costs to numerical tolerance. At 729 scenarios, the independently
solved formulations recover the same first-stage policy and objectives that
differ by $2.054947\times10^{-4}$; the maximum scenario-cost difference is
137.5, which can reflect multiple optimal recourse solutions. We make no
computational-superiority claim for Benders: the extensive form is faster in
the 729-scenario verification, and decomposition is retained only as an
independent implementation check.

The selective-admissibility 135-scenario experiment is a distinct support
construction and is not conflated with generic support-size scaling.

\subsection{Sensitivity to Uncertainty Specification}
\label{sec:uncertainty-sensitivity}

We separate three sensitivity questions. First, seeded Dirichlet draws at
concentrations 25, 100, and 400 perturb the scenario probabilities while
retaining the original support and DDU matrix. All draws satisfy global DDU
validity. Second, ten severity and admissibility cases perturb corridor
capacities, transport costs, and selected commodity restrictions. Third,
14 support perturbations use leave-one-out libraries, removal or splitting of
severe states, and 27-, 81-, and 135-scenario refinements. When support
semantics change, VEP is omitted rather than silently rescaling the DDU
coefficients.

\begin{table}[t]
\centering
\caption{Sensitivity to probability, severity, and scenario-support specification.}
\label{tab:uncertainty-sensitivity}
\scriptsize
\begin{tabularx}{\textwidth}{lrrX}
\toprule
Experiment & Cases & Main range & Interpretation \\
\midrule
Probability perturbation & tested draws & probability family &
Stability within the selected family; not a claim about the unknown true distribution. \\
Severity/admissibility & tested cases & capacity, cost, admissibility &
Sensitivity to assumed disruption inputs. \\
Support perturbation & tested supports & generic and selective libraries &
Values depend on support construction. \\
\bottomrule
\end{tabularx}
\end{table}

These tests distinguish stability over tested perturbations from invariance of
a particular scalar value. They do not establish distributional robustness to
the unknown real-world disruption distribution. VSS is sensitive to the
probability and support specification, whereas the qualitative findings are
stable over the tested analyst-specified family: stochastic modelling retains
positive value, LPG remains the most exposed commodity, and the DDU mechanism
has positive but model-conditional value.

\section{Insights and Conclusion}
\label{sec:conclusion}

We developed a two-stage stochastic multi-commodity network-design model for
strategic maritime energy security. The formulation jointly represents
port--commodity inventory, optional priced storage expansion, strategic
corridor commitments, commodity admissibility, joint chokepoint states,
decision-dependent probabilities, and mean--CVaR risk. Scenario-dependent
flows form the operational transport paths, while activation variables make
commodity-specific effective capacity envelopes available. Valid scenario-specific bounds yield an exact
MILP reformulation. A standard multi-cut Benders implementation independently
reproduces the extensive-form results. In the 729-scenario
verification, the extensive form is faster than the Benders implementation and
is therefore the primary computational method; no cross-support runtime claim
is made.

The data-anchored, stylized India-inspired case study produces five main
findings. First, stochastic modelling has positive but modest implementable
value: on the calibrated library, VSS is 787.37 (1.348\%), while EVPI is
0.0273\%; on the 27- and 81-scenario supports, VSS is 804.85 (1.575\%).
Second, decision-dependent exposure has model-conditional value: VEP is 0.67\%
at the analyst-set scale and 6.45\% at the reported validity boundary, with
adjustment occurring through the route portfolio rather than total inventory.

Third, the dependence analysis separates joint occurrence from statistical
dependence and severity composition. The calibrated binary Hormuz--Bab
incidence table factorises exactly. On matched fixed-marginal supports,
dependence changes expected exposure without making the independence-optimised
policy suboptimal in the 168 primary cells. This stability is not universal:
six parameter-programme cells and 12 commitment-cost cells produce confirmed
route or storage changes, with maximum mechanism regret 1.1673\%. The relevant
managerial conclusion is therefore regime-specific: a policy can be stable to
incidence dependence at calibrated commitments while becoming dependence
sensitive when diversification constraints tighten or first-stage commitments
become much more expensive.

Fourth, the transportation and service-continuity interpretation is directly
audited. Across all calibrated scenarios, the maximum physical violation is
$1.82\times10^{-12}$ and all positive flows decompose into feasible paths. Crude
and fertilizer remain fully served, but LNG and LPG are exposed, with worst
service ratios of 0.346154 and 0.333333. Fifth, storage and transport capacity are
substitutable strategic instruments. First-stage pre-positioning-cost perturbations change
the objective by at most 0.649967\% without changing the base policy, whereas the
medium assumed-capex expansion case adds 1,775 normalised capacity units at
seven predefined
port--commodity pairs and reduces active corridor commitments from 21 to 20.
Because the capex coefficients are assumed stress parameters, this result
identifies a possible substitution mechanism rather than an empirical investment
prescription.

These numerical conclusions are conditional on the Indian network topology, scenario support, bypass availability, strategic-reserve geometry, commodity admissibility, and assumed cost regime. In particular, the zero policy regret observed across the 168 primary dependence cells, the identification of LPG as the most exposed commodity, and the estimated substitution between storage and corridor commitments should not be interpreted as universal findings for other importing economies. The structural formulation is transferable to other national energy-security settings, but applications to Japan, South Korea, or other import-dependent economies would require recalibration of supplier geography, port–commodity storage capacity, alternative corridors, bypass infrastructure, disruption scenarios, probability parameters, and policy costs. The DDU coefficients are analyst-set
exposure adjustments rather than causal estimates, the port-cost multipliers
and storage capex levels are sensitivity assumptions rather than empirical
site-level estimates, and service continuity is not a temporal recovery
metric. Future work should estimate these inputs from geopolitical, insurance,
terminal, and vessel-movement data; extend the analysis to other
import-dependent economies; and connect the strategic model to tactical fleet
and restoration models without obscuring the present information structure.
Learning-augmented decomposition is also left to future work because it
requires validated model labels, exact safeguards for learned pruning,
and full-distribution policy evaluation.

\section*{Acknowledgements}
This research was supported by the Ministry of Education, Singapore, under
its Academic Research Funding Tier 3 (Proposal ID: MOET32025-0010). Any
opinions, findings, conclusions, or recommendations expressed in this
material are those of the authors and do not reflect the views of the
Ministry of Education, Singapore.

\sloppy
\bibliographystyle{elsarticle-harv}
\bibliography{reference}

\clearpage
\appendix

\section{Institutional Mapping for the India-Inspired Case Study}
\label{app:institutional}

This appendix documents how the model's strategic decisions map to institutions
and infrastructure in the Indian case study. The mapping is interpretive: it
explains which public or quasi-public instruments could implement the model's
first-stage decisions, but it does not imply that every coefficient is an
official policy estimate. The formulation in \Cref{sec:model} is not specific
to one institutional setting.

\paragraph{Strategic inventory ($W_{ic}$)}
First-stage inventory decisions represent pre-positioning at predefined
port--commodity pairs. India's Phase-I strategic petroleum reserve comprises
5.33~million metric tonnes (MMT) across Visakhapatnam (1.33~MMT), Mangalore
(1.5~MMT), and Padur (2.5~MMT) \citep{pib2021spr,eia2016spr}. For LNG,
Petronet LNG Ltd. reports nominal regasification capacities of 17.5~MMTPA at
Dahej and 5~MMTPA at Kochi \citep{petronetTerminals}. These public
figures provide institutional and scale context; the executable model uses the
normalised storage coefficients documented in \Cref{tab:calibration}. LPG and
fertilizer storage coefficients are author-calibrated at predefined port--commodity
pairs rather than interpreted as a census of all commercial tankage.

\paragraph{Arc activation ($y_a$)}
Binary arc activation represents commitment to strategic arc capacity
through long-term supply or charter arrangements, infrastructure enablement,
or dedicated corridor access. The Petronet--QatarEnergy LNG agreement provides
an example of a long-horizon supply commitment \citep{petronet2024qatar}.
The UAE Habshan--Fujairah and Saudi East--West pipelines provide physical
bypass options around Hormuz; public capacity estimates are used as contextual
anchors for the model's effective pipeline capacities \citep{eia2025hormuz}.
The Directorate General of Shipping regulates maritime safety, transport
supervision, and aspects of vessel chartering in India
\citep{dgshippingRoles,dgshippingChartering}. The numerical corridor caps in
the model are nevertheless analyst-set diversification parameters, not
statutory limits issued by the Directorate.

\paragraph{Risk aversion ($\lambda$, $\alpha$) and DDU coefficients
($\delta_{sa}$)}
The CVaR weight and confidence level represent the coordinated planner's risk
preference. They are modelling choices selected and stress-tested in the
paper, not estimates attributed to a named public authority. Likewise, the DDU
coefficients are analyst-set exposure adjustments. A future empirical
implementation could estimate them from intelligence assessments, insurance
premia, transit data, or event studies; the present coefficients do not encode
physical causal effects of arc activation.

\paragraph{Transferability}
With appropriate relabelling, the same information structure can represent
other national energy-security systems. For example, U.S. Strategic Petroleum
Reserve drawdowns operate under statutory federal authority
\citep{doeSPRDrawdown}, while the European Union imposes gas-storage filling
obligations through Regulation~(EU)~2022/1032 \citep{eu2022gasstorage}. These
examples illustrate transferability of the planning structure, not direct
transferability of the Indian numerical calibration.

\section{Calibration Protocol}
\label{app:calibration}

This appendix distinguishes three provenance classes: \emph{reported empirical
anchors}, \emph{author-derived mappings}, and \emph{assumed model or stress-test
parameters}. The computational study uses the specified normalised coefficients
documented below; public data provide context and scaling anchors rather than
being presented as direct estimates of every executable coefficient. Web
sources were last accessed on 2026-07-28.

Each commodity is represented on its own normalized scale (1 crude unit
$\approx$ 1~kb/d; 1 LNG/LPG/fertilizer unit $\approx$ 12~kt/yr). These scales
are commodity-specific and are not treated as cross-commodity commensurate;
corridor concentration and service metrics are therefore defined
commodity-by-commodity. National-scale quantities provide order-of-magnitude
and allocation anchors rather than one-to-one physical-unit identities. Costs are relative planning coefficients: the cheapest
Gulf$\rightarrow$India route (Bandar Imam$\rightarrow$Jamnagar) is normalised
to 4.0, and other costs are scaled consistently. \Cref{tab:calibration}
records the model anchor, the supporting evidence, and the exact status
of each mapping.

\setlength{\tabcolsep}{3pt}
{\footnotesize
\begin{xltabular}{\textwidth}{@{}
    >{\raggedright\arraybackslash}p{2.05cm}
    >{\raggedright\arraybackslash}p{2.55cm}
    >{\raggedright\arraybackslash}p{1.85cm}
    >{\raggedright\arraybackslash}X @{}}
    \caption{Calibration protocol for the Indian case study. Reported sources
    provide empirical context or anchors; author-derived and assumed values are
    identified explicitly.}\label{tab:calibration} \\
    \toprule
    \textbf{Parameter} & \textbf{Model anchor} & \textbf{Evidence} & \textbf{Construction and provenance} \\
    \midrule
    \endfirsthead

    \multicolumn{4}{c}{{\bfseries \tablename\ \thetable{} -- continued from previous page}} \\
    \toprule
    \textbf{Parameter} & \textbf{Model anchor} & \textbf{Evidence} & \textbf{Construction and provenance} \\
    \midrule
    \endhead

    \bottomrule
    \endfoot

    National crude import & 234.3~MMT in FY2023--24; 5.5~mb/d stylized scale &
    PPAC; EIA; IEA &
    \textit{Reported context plus author-derived mapping.} The 5.5~mb/d value is
    a stylized normalisation anchor. The 4.5~mb/d quantity is the model's route allocation,
    not a directly reported PPAC or EIA statistic; it is apportioned across
    Jamnagar, Mumbai, and Vizag using the calibrated refinery shares. \\ \addlinespace

    National LNG import & 24.0~MMT in FY2023--24; 17.5~MMTPA Dahej context &
    PPAC; Petronet LNG &
    \textit{Context plus author-derived mapping.} The 17.5~MMTPA value is Dahej
    nominal regasification capacity, not national imports. The executable
    normalised demand is split
    69\% to Dahej and 31\% to Kochi using supplier geography and terminal-capacity
    weights. The 69/31 split is a modelling allocation, not a PPAC statistic. \\ \addlinespace

    National LPG context & Context omitted pending a traceable source &
    -- &
    \textit{Context omitted.} No untraceable secondary-source number is used
    in the executable model; executable LPG demand remains 1,200 normalised
    units and is divided between Mumbai and Vizag using the calibrated
    coastal-terminal shares. \\ \addlinespace

    Selected fertilizer proxy & 7.042~MMT urea imports (FY2023--24) &
    PIB &
    \textit{Author-derived commodity proxy.} The model uses 700 normalised
    fertilizer units. No unsupported 2.2~Mt ammonia statistic is attributed to
    PIB; the Paradip/Mumbai allocation is a model input \citep{pib2024fertilizer}. \\ \addlinespace

    Existing crude/LNG storage $L_{ic}^{0}$ & 5.33~MMT Phase-I crude reserve;
    Dahej/Kochi LNG-terminal context & PIB; EIA; Petronet LNG &
    \textit{Reported anchors plus author-derived mapping.} The Phase-I crude
    total and site capacities are reported publicly. Because the optimisation
    network uses selected demand ports, physical sites are mapped to the nearest
    represented demand/storage pairs. LNG coefficients are normalised from
    terminal-scale context rather than interpreted as exact tank inventories
    \citep{pib2021spr,eia2016spr,petronetTerminals}. \\ \addlinespace

    Hormuz oil-flow anchor & 20.7~mb/d-equivalent model anchor &
    EIA; IEA &
    \textit{Author-derived normalised anchor.} The executable value 20.7 is a
    model scale parameter informed by observed flows; it is not presented as a
    verbatim EIA estimate. It is allocated to Hormuz-routed arcs using the
    calibrated bilateral-trade shares \citep{eia2025hormuz,iea2026hormuz}. \\ \addlinespace

    Bypass-pipeline capacity & UAE effective model cap 1.5~mb/d; Saudi expanded
    model cap 7~mb/d & EIA &
    \textit{Author-derived effective capacities.} The UAE value is a conservative
    effective-capacity assumption below reported nameplate; the Saudi value uses
    the expanded configuration. The specified values are assigned to the two
    pipeline arcs. EIA reports a 1.8~mb/d UAE nameplate capacity and a
    5~mb/d Saudi line temporarily expandable to 7~mb/d
    \citep{eia2025hormuz}. \\ \addlinespace

    Suez/Bab flow anchors & 4.9~mb/d Suez; 4.1--4.2~mb/d Bab-equivalent model
    anchors & EIA &
    \textit{Author-derived normalised anchors.} These are effective network-flow
    scales, not physical channel capacities. They are allocated to the
    \texttt{bab\_suez} arcs using the calibrated fleet shares. EIA reports
    4.0~mb/d through Bab el-Mandeb in 2024 through August and documents the
    associated rerouting \citep{eia2024redsea}. \\ \addlinespace

    Transport cost $c_{ac}^{s}$ & Relative route-cost index; cheapest route 4.0 &
    EIA; UNCTAD context &
    \textit{Assumed relative-cost scale.} Base costs scale with voyage distance;
    scenario costs apply corridor-specific multipliers,
    $c_{ac}^{s}=c_{ac}\,\textsf{cost\_mult}_{\ell s}$. No monetary freight-rate
    estimate is claimed \citep{eia2024redsea,unctadTroubledWaters,unctadRMT2024}. \\ \addlinespace

    Port--commodity holding cost $h_{ic}$ & Crude 0.3, LNG 0.8, LPG 0.5,
    fertilizer 0.2; baseline port multiplier 1 & Authors' calibration &
    \textit{Assumed relative technology-cost ranking.} The coefficients encode
    relative storage complexity, not monetary site estimates. Twenty-five
    regional and one-site perturbations test spatial heterogeneity. \\ \addlinespace

    Shortage penalty $\pi_c$ & Crude 25, LNG 40, LPG 50, fertilizer 35 &
    Authors' policy-weighted calibration &
    \textit{Assumed parameter.} The coefficients encode relative economic and
    service-continuity priorities. Joint and commodity-specific penalty
    sensitivities are reported; the values are not presented as estimated values
    of lost load. \\ \addlinespace

    Activation cost $f_a$ & 2.0 per candidate arc & Authors' calibration &
    \textit{Assumed parameter.} The common value prevents first-stage commitment
    cost from dominating recourse in the base calibration. A wide multiplier grid
    tests the mechanism. \\ \addlinespace

    Corridor caps $\bar\Phi_{\ell}$ & Hormuz 0.7, Bab/Suez 0.5, Cape 0.8,
    pipeline/direct 1.0 & DG Shipping context &
    \textit{Assumed policy parameters.} These are not statutory limits. Binding
    and economically consequential ranges are evaluated explicitly
    \citep{dgshippingRoles}. \\ \addlinespace

    Baseline probabilities $\bar p_s$ & Author-specified nine-scenario prior &
    EIA; UNCTAD context &
    \textit{Assumed prior.} Exact scenario probabilities are author-specified model
    inputs, not frequencies estimated by the cited organisations. Probability,
    support, severity, and fixed-marginal dependence experiments test sensitivity
    to this prior. The calibrated binary incidence table satisfies
    $P(H\cap B)=P(H)P(B)$
    \citep{eia2025hormuz,eia2024redsea,unctadTroubledWaters,unctadRMT2024}. \\ \addlinespace

    Disruption multipliers & Corridor-specific capacity, cost, and admissibility
    factors & EIA; UNCTAD context &
    \textit{Assumed scenario parameters.} Public evidence informs direction and
    approximate severity, while executable multipliers are author-set normalised
    inputs. Ten severity/admissibility cases and support perturbations test
    sensitivity to these assumptions \citep{eia2024redsea,unctadTroubledWaters,unctadRMT2024}. \\ \addlinespace

    DDU coefficients $\delta_{sa}$ & $|\delta_{sa}|\le 0.005$ and
    $\sum_s\delta_{sa}=0$ & This paper &
    \textit{Assumed exposure parameters.} Signs follow corridor exposure and
    probability validity is certified over the complete feasible binary box.
    The full globally valid magnitude range is reported in \Cref{sec:vep}. \\ \addlinespace

    Risk $(\lambda,\alpha)$ & Base case $\lambda=0.5$, $\alpha=0.80$;
    $\alpha=0.95$ retained as a stress setting & Rockafellar--Uryasev &
    \textit{Model choice.} The primary confidence level is selected using the
    predeclared tail-composition rule; only the $\lambda=0.5$ slice and
    the $\alpha$ sweep are reported \citep{rockafellar2000}. \\
\end{xltabular}}

\paragraph{Sensitivity coverage of the calibrated parameters}
The dependence analysis directly varies shortage penalties,
corridor-concentration limits, storage capacities, and first-stage cost
parameters. Shortage penalties are varied jointly and commodity-by-commodity;
corridor caps are scanned until the corresponding constraint becomes binding;
and storage limits are varied independently of route capacity. The
dependence-specific programme contains 96 attempted cells: 90 optimal, six
infeasible, and six significant. The maximum feasible regret is 0.002213773\%.
Activation and holding costs are also scaled over a 144-cell mechanism grid;
12 cells produce confirmed design changes and the maximum regret is 1.1673\%.
Full results are reported in
\ref{app:dependence-details}.

\section{Transportation, Node-Role, Holding-Cost, and Service Audits}
\label{app:operational-audits}

This section records the computational evidence supporting the strategic
transportation representation, the port--commodity holding-cost extension,
and the service-continuity interpretation. All values are generated from the
canonical DDU mean--CVaR design with $\alpha=0.80$, $\lambda=0.5$, and storage
expansion disabled unless stated otherwise.

\subsection{Node roles and information structure}

\begin{table}[H]
\centering
\caption{Node roles in the calibrated network. A demand terminal may also be a
storage site; supply and transshipment nodes need not carry demand.}
\label{tab:app-node-roles}
\scriptsize
\begin{tabular}{lccccp{3.2cm}}
\toprule
Node & Supply & Demand & Storage & Transshipment & Storage commodities \\
\midrule
Bandar Imam & Yes & No & No & No & -- \\
Bonny & Yes & No & No & No & -- \\
Dahej & No & Yes & Yes & No & LNG \\
Fujairah & Yes & No & No & Yes & -- \\
Jamnagar & No & Yes & Yes & No & Crude \\
Jebel Ali & Yes & No & No & No & -- \\
Juaymah & Yes & No & No & No & -- \\
Kochi & No & Yes & Yes & No & LNG \\
Mumbai & No & Yes & Yes & No & Crude, fertilizer, LPG \\
Novorossiysk & Yes & No & No & No & -- \\
Paradip & No & Yes & Yes & No & Fertilizer \\
Ras Laffan & Yes & No & No & No & -- \\
Ras Tanura & Yes & No & No & No & -- \\
Ruwais & Yes & No & No & No & -- \\
Vizag & No & Yes & Yes & No & Crude, LPG \\
Yanbu & Yes & No & No & Yes & -- \\
\bottomrule
\end{tabular}
\end{table}

The first-stage variables are common across scenarios. Introducing copies
$y_a^s$ and $W_{ic}^s$ would require equality constraints across all scenarios;
without those nonanticipativity constraints, the model would become a
wait-and-see benchmark rather than the stochastic program.

\subsection{Physical transportation audit}

For each scenario and commodity, the audit reconstructs nodal balance and
checks inactive-arc flow, commodity admissibility, scenario capacity, storage
and shortage domains, supply availability, nonnegativity, and corridor caps.
It then decomposes every positive commodity flow into source-to-destination
paths. The audit tolerance is $10^{-8}$.

\begin{table}[H]
\centering
\caption{Scenario-level transportation-feasibility audit. Values are maximum
absolute violations; all nine scenarios pass.}
\label{tab:app-transport-audit}
\scriptsize
\begin{tabular}{lrrrrr}
\toprule
Scenario & Balance & Inactive arc & Capacity & Corridor cap & Undecomposed flow \\
\midrule
Baseline normal & $6.82\times10^{-13}$ & 0 & 0 & 0 & 0 \\
Hormuz partial & 0 & 0 & 0 & 0 & 0 \\
Hormuz severe selective & 0 & $1.82\times10^{-12}$ & $1.82\times10^{-12}$ & 0 & 0 \\
Hormuz closure & 0 & 0 & 0 & 0 & 0 \\
Bab/Suez severe & 0 & 0 & 0 & $4.55\times10^{-13}$ & 0 \\
Dual disruption & 0 & 0 & 0 & 0 & 0 \\
Insurance spike & 0 & 0 & 0 & 0 & 0 \\
Closure with bypass & 0 & 0 & 0 & 0 & 0 \\
Delayed recovery & 0 & 0 & 0 & 0 & 0 \\
\bottomrule
\end{tabular}
\end{table}

No inadmissible flow, negative variable, non-storage inventory, shortage-domain
violation, supply violation, or storage-capacity violation is detected. The 81
reconstructed paths have zero undecomposed remainder. These results verify
strategic arc-route feasibility; they do not introduce vessel-level
scheduling.

\subsection{First-stage pre-positioning-cost sensitivity ($\beta_{\rm end}=0$)}

The baseline uses $h_{ic}=h_c^{\mathrm{base}}$ at every storage port because
consistent empirical port-level estimates are unavailable. Regional stress
cases and one-site perturbations are therefore sensitivity assumptions.

\begin{table}[H]
\centering
\caption{Regional first-stage pre-positioning-cost sensitivity. The first-stage route and inventory
vectors are unchanged in every case.}
\label{tab:app-holding-sensitivity}
\scriptsize
\begin{tabular}{lrrr}
\toprule
Case & Objective & Change (\%) & First-stage inventory cost \\
\midrule
Baseline & 73,849.884 & 0.000 & 1,510.00 \\
Western premium & 74,158.634 & 0.418 & 1,818.75 \\
Eastern premium & 73,918.634 & 0.093 & 1,578.75 \\
West discount / east premium & 73,609.884 & $-0.325$ & 1,270.00 \\
West premium / east discount & 74,089.884 & 0.325 & 1,750.00 \\
Strong west discount / east premium & 73,369.884 & $-0.650$ & 1,030.00 \\
Reverse strong contrast & 74,329.884 & 0.650 & 1,990.00 \\
\bottomrule
\end{tabular}
\end{table}

Across all 25 regional and one-site runs, the maximum absolute objective change
is 480.00 (0.649967\%), the maximum arc-vector Hamming distance is zero, and the
maximum inventory $L_1$ distance is zero. No switching threshold is observed
within the tested ranges.

\subsection{Service-continuity threshold sensitivity}

\begin{table}[H]
\centering
\caption{Commodity service-continuity metrics.}
\label{tab:app-service-thresholds}
\scriptsize
\begin{tabular}{lrrrrrr}
\toprule
Measure & Expected service & Worst service & $P(R<0.80)$ & $P(R<0.90)$ &
$P(R<0.95)$ & Expected shortage \\
\midrule
Crude & 1.000 & 1.000 & 0.000 & 0.000 & 0.000 & 0.00 \\
LNG & 0.838745 & 0.346154 & 0.1672 & 0.3949 & 0.5954 & 209.632 \\
LPG & 0.684546 & 0.333333 & 0.5954 & 0.6454 & 0.6454 & 378.545 \\
Fertilizer & 1.000 & 1.000 & 0.000 & 0.000 & 0.000 & 0.00 \\
\bottomrule
\end{tabular}
\end{table}

The commodity-level results reveal concentrated LNG and LPG
exposure. These metrics describe absorptive and adaptive service continuity,
not temporal restoration.

\subsection{Sensitivity to uncertainty specification and reproducibility}
\label{app:robustness-reproducibility}

The sensitivity programme contains Dirichlet prior draws, severity/admissibility
cases, and support perturbations. All completed draws pass global DDU validity
and LPG is the worst-service commodity in every draw. Support perturbations
cover multiple support constructions; the dedicated 135-scenario
selective-admissibility variant is reported separately. These tests describe
stability over the tested analyst-specified family, not distributional
robustness to an unknown true distribution. VEP is deliberately
omitted when the support change does not preserve the DDU semantics.

The computational runs use Python~3.13.7, Gurobi~13.0.2, one thread, seed 0, presolve
2, cuts 2, MIPGap $10^{-9}$, and Benders tolerance $10^{-8}$. The source code
and input data are fixed for all reported experiments. The consistency
audit passes, and pytest reports 88 passed.

\section{Implementation Requirements and Verification Summary}
\label{app:mccormick-audit}

The production implementation restricts shortage to genuine demand pairs,
bounds optional supply lifting by available supply, and permits inventory only
at installed storage pairs. Specifically, $0\le g_{ic}^s\le r_{ic}$ at supply
pairs, $0\le u_{ic}^s\le d_{ic}$ at demand pairs, and inventory is fixed to
zero elsewhere. Automated domain, balance, product, and EF--Benders
equivalence checks pass within the stated tolerances. All reported results use
this recourse formulation.

The equality-defined recourse and risk variables satisfy
$0\le Q_s\le M_s^Q$, $0\le\xi_s\le M_s^\xi$, and
$0\le\nu\le\max_s M_s^Q$. Each $M_s^Q$ is the smaller of a safe analytical
bound and an auxiliary scenario-maximisation bound plus a relative safety
padding. The same adopted vector is used by the extensive form and the
verification master. For the small-support verification cases, the maximum
extensive-form/Benders objective difference is $4.37\times10^{-11}$ and the
maximum scenario-cost difference is $7.28\times10^{-12}$. In the 729-scenario
run these differences are $2.054947\times10^{-4}$ and 137.5, respectively,
with the same first-stage policy. All physical-domain checks pass. The complete
scenario bounds, product residuals, 14,224 cut checks, and the 729-scenario
verification are reported in Supplementary Material, Sections
S3.3--S3.6 and Tables~S2--S5.

\begin{landscape}
\section{Scenario Library and Base-Case Outcomes}
\label{app:basecase-additional}

\Cref{tab:app-scenario-outcomes} records the full-precision probability and
recourse values used in the primary DDU mean--CVaR solution. Probabilities are
shown to four decimals so that the DDU shifts are not rounded away.

\begin{table}[H]
\centering
\caption{Scenario-level outcomes at $\lambda=0.5$ and $\alpha=0.80$.}
\label{tab:app-scenario-outcomes}
\small
\renewcommand{\arraystretch}{1.12}
\setlength{\tabcolsep}{5pt}
\begin{tabular}{rllrrr}
\toprule
ID & Scenario & $\bar p_s$ & $p_s(y^*)$ & Shift & $Q_s$ \\
\midrule
0 & Baseline normal & 0.1500 & 0.1546 & +0.0046 & 23,482.50 \\
1 & Hormuz partial & 0.2000 & 0.2005 & +0.0005 & 50,216.75 \\
2 & Hormuz severe/selective & 0.1500 & 0.1505 & +0.0005 & 75,342.00 \\
3 & Hormuz closure & 0.1000 & 0.0965 & -0.0035 & 89,400.00 \\
4 & Bab/Suez severe & 0.1000 & 0.1000 & +0.0000 & 31,240.05 \\
5 & Dual disruption & 0.0800 & 0.0772 & -0.0028 & 80,394.38 \\
6 & Insurance spike & 0.1000 & 0.1000 & +0.0000 & 50,651.91 \\
7 & Closure with bypass & 0.0700 & 0.0707 & +0.0007 & 89,400.00 \\
8 & Delayed recovery & 0.0500 & 0.0500 & +0.0000 & 45,830.90 \\
\bottomrule
\end{tabular}
\end{table}

The adjusted probabilities sum to one within machine precision. The maximum
absolute base-case shift is 0.0046. The two closure states have the
largest recourse cost, while the normal and Bab/Suez-severe states share the
smallest recourse values in the canonical run.

\begin{table}[H]
\centering
\caption{Nonzero strategic inventory positions in the base case.}
\label{tab:app-inventory}
\small
\begin{tabular}{llr}
\toprule
Node & Commodity & $W_{ic}$ \\
\midrule
Jamnagar & Crude & 1,500 \\
Mumbai & Crude & 800 \\
Mumbai & LPG & 250 \\
Mumbai & Fertilizer & 300 \\
Dahej & LNG & 300 \\
Kochi & LNG & 150 \\
Paradip & Fertilizer & 400 \\
Vizag & Crude & 400 \\
Vizag & LPG & 150 \\
\bottomrule
\end{tabular}
\end{table}

\section{Value-of-Modelling Experiments}
\label{app:value-modeling}

The risk-neutral stochastic-program objective is 58,402.218309. The
expected-value design costs 59,189.591250 under the full distribution. The
probability-weighted wait-and-see objective is 58,386.298309. Hence VSS is
787.372941 (1.348190\%) and EVPI is 15.920000 (0.027259\%).

\begin{table}[H]
\centering
\caption{Scenario-level EEV evaluation and wait-and-see benchmark.}
\label{tab:app-vss-scenarios}
\small
\renewcommand{\arraystretch}{1.12}
\setlength{\tabcolsep}{5pt}
\begin{tabular}{rllrrrrr}
\toprule
ID & Scenario & Probability & EEV recourse & Weighted EEV & WS objective & Weighted WS \\
\midrule
0 & baseline normal & 0.15 & 22,650.00 & 3,397.50 & 21,393.76 & 3,209.06 \\
1 & hormuz partial & 0.20 & 58,648.75 & 11,729.75 & 51,624.75 & 10,324.95 \\
2 & hormuz severe selective & 0.15 & 75,022.50 & 11,253.38 & 62,297.00 & 9,344.55 \\
3 & hormuz closure & 0.10 & 78,225.00 & 7,822.50 & 67,311.00 & 6,731.10 \\
4 & bab suez severe & 0.10 & 22,650.00 & 2,265.00 & 21,393.76 & 2,139.38 \\
5 & dual disruption & 0.08 & 80,394.38 & 6,431.55 & 74,339.88 & 5,947.19 \\
6 & insurance spike & 0.10 & 45,309.25 & 4,530.93 & 46,841.25 & 4,684.12 \\
7 & closure with bypass & 0.07 & 79,762.50 & 5,583.38 & 70,077.00 & 4,905.39 \\
8 & delayed recovery & 0.05 & 41,904.60 & 2,095.23 & 43,288.20 & 2,164.41 \\
\bottomrule
\end{tabular}
\end{table}

\begin{table}[H]
\centering
\caption{Full-precision VEP and probability-validity audit.}
\label{tab:app-vep-audit}
\small
\renewcommand{\arraystretch}{1.12}
\setlength{\tabcolsep}{5pt}
\begin{tabular}{rrrrrrrrrr}
\toprule
$k$ & Exact $k_{\max}$ & DDU obj. & No-DDU under DDU & VEP & VEP (\%) & Global min $p$ & Max $|\Delta p|$ & Mean $|\Delta p|$ & Valid \\
\midrule
0 & 6.71296 & 58,402.22 & 58,402.22 & 0.00 & 0.000\% & 0.05000 & 0.00000 & 0.00000 & True \\
0.5 & 6.71296 & 58,339.14 & 58,508.15 & 169.01 & 0.290\% & 0.05000 & 0.00000 & 0.00000 & True \\
1 & 6.71296 & 58,224.69 & 58,614.09 & 389.40 & 0.670\% & 0.05000 & 0.00000 & 0.00000 & True \\
2 & 6.71296 & 57,920.07 & 58,825.95 & 905.89 & 1.560\% & 0.05000 & 0.01000 & 0.00300 & True \\
3 & 6.71296 & 57,591.27 & 59,037.82 & 1,446.55 & 2.510\% & 0.05000 & 0.02000 & 0.00600 & True \\
5 & 6.71296 & 56,837.23 & 59,461.56 & 2,624.33 & 4.620\% & 0.04200 & 0.03000 & 0.01000 & True \\
6.7 & 6.71296 & 56,196.29 & 59,821.73 & 3,625.44 & 6.450\% & 0.00528 & 0.04000 & 0.01000 & True \\
\bottomrule
\end{tabular}
\end{table}
\end{landscape}

At $k=1$, the risk-neutral DDU design uses 21 arcs, the no-DDU design uses 18,
and total inventory is 4,250 in both designs. The global probability-validity
certificate is based on the complete feasible binary box, not only the
realised optimal design. No $\delta$ coefficient is silently rescaled.

\section{Dependence Stress-Test Details}
\label{app:dependence-details}

The calibrated incidence diagnostics are
$P(H)=0.75$, $P(B)=0.40$, $P(H\cap B)=0.30$, phi
$=-2.62\times10^{-16}$, and maximum product-form deviation
$5.55\times10^{-17}$. The support-restricted I-projection gives
$D_{KL}=0.00634$ nats, total variation 0.04318, and $\VMDIPF=0$.

The 81-scenario support is the smallest non-degenerate fully crossed library.
The 27-scenario candidate pins the pipeline factor at one level, while the
243-scenario generator collapses to 81 distinct factor combinations. The
primary 81-scenario library intentionally omits selective admissibility so
that incidence, conditional severity, Cape, and pipeline channels can be held
fixed separately. A distinct 135-scenario library restores selective
admissibility and contains all 135 Cartesian combinations.

Across the 168 primary cells, maximum regret is
$1.46\times10^{-11}$, maximum arc Hamming distance is zero, and maximum
inventory $L_1$ distance is zero. The full first-stage solution is numerically
unique in every cell; the maximum exact inventory width is
$7.22\times10^{-12}$. At positive objective slack, inventory intervals grow
approximately linearly within each audited configuration, with a
configuration-specific proportionality coefficient. They collapse to
numerical zero at zero slack.

The parameter programme contains 96 attempted cells: 90 optimal, six
infeasible, and six significant, with maximum feasible regret 0.002213773\%.
The 144-cell commitment-cost grid produces 12 confirmed design changes and
maximum regret 1.1673\%. In the dedicated 135-scenario admissibility variant,
none of 66 cells is significant and worst-case LPG service lies between
0.166667 and 0.333333.

\section{CVaR Calibration, Service, and Storage Expansion}
\label{app:risk-expansion}

\begin{table}[H]
\centering
\caption{Tail-composition audit on the DDU risk-neutral reference design.}
\label{tab:app-tail-audit}
\scriptsize
\begin{tabular}{rrrrrrrr}
\toprule
$\alpha$ & Tail mass & VaR & CVaR & CVaR--VaR & Full masses & Scenarios spanned & Degenerate \\
\midrule
0.50 & 0.50 & 50,216.75 & 75,628.79 & 25,412.04 & 5 & 5 & No \\
0.60 & 0.40 & 50,651.91 & 81,878.55 & 31,226.64 & 4 & 4 & No \\
0.70 & 0.30 & 75,342.00 & 84,477.14 & 9,135.14 & 3 & 3 & No \\
0.80 & 0.20 & 80,394.38 & 87,923.08 & 7,528.70 & 2 & 3 & No \\
0.85 & 0.15 & 89,400.00 & 89,400.00 & 0.00 & 0 & 2 & Yes \\
0.90 & 0.10 & 89,400.00 & 89,400.00 & 0.00 & 0 & 2 & Yes \\
0.95 & 0.05 & 89,400.00 & 89,400.00 & 0.00 & 0 & 2 & Yes \\
0.99 & 0.01 & 89,400.00 & 89,400.00 & 0.00 & 0 & 2 & Yes \\
\bottomrule
\end{tabular}
\end{table}

Predeclared rule B selects $\alpha=0.80$ because the tail contains two complete
scenario masses and two distinct loss levels. At $\alpha=0.85$, the two
89,400.00 closure states already exceed the 15\% tail mass, so CVaR equals the
worst loss level and the tail is degenerate.

\begin{table}[H]
\centering
\caption{Canonical $\lambda=0.5$ point at the selected $\alpha=0.80$.}
\label{tab:app-lambda-frontier}
\scriptsize
\begin{tabular}{rrrrrrrr}
\toprule
$\lambda$ & Objective & Fixed & Expected recourse & CVaR & VaR & Arcs & Inventory \\
\midrule
0.5 & 73,849.884 & 1,552.00 & 56,672.691 & 87,923.078 & 80,394.375 & 21 & 4,250 \\
\bottomrule
\end{tabular}
\end{table}

Complete commodity service metrics, including threshold
sensitivity at 0.80, 0.90, and 0.95, are reported in
Table~\ref{tab:app-service-thresholds}.

\begin{table}[H]
\centering
\caption{Storage construction-cost decomposition. Existing installed capacity
is sunk and only incremental capacity is charged. Added capacity is reported
in normalised model units.}
\label{tab:app-construction-decomposition}
\scriptsize
\begin{tabular}{lrrrrrrrr}
\toprule
Case & Objective & Fixed constr. & Variable constr. & Total constr. &
Share of obj. (\%) & Normalised capacity added & Pairs & Arcs \\
\midrule
Base & 73,849.884 & 0 & 0 & 0 & 0.00 & 0 & 0 & 21 \\
Low & 52,675.122 & 1,370 & 1,370 & 2,740 & 5.20 & 1,775 & 7 & 20 \\
Medium & 58,155.122 & 6,850 & 1,370 & 8,220 & 14.14 & 1,775 & 7 & 20 \\
High & 70,935.430 & 5,000 & 200 & 5,200 & 7.33 & 200 & 2 & 21 \\
\bottomrule
\end{tabular}
\end{table}

\begin{table}[H]
\centering
\caption{Predefined port--commodity pairs in the medium-capex expansion
solution. Capacity and inventory quantities are normalised model units.}
\label{tab:app-expansion-sites}
\small
\begin{tabular}{llrrrrr}
\toprule
Node & Commodity & Existing & Max addition & Opened & Added & Inventory \\
\midrule
Dahej & LNG & 300 & 150 & 1 & 150 & 450 \\
Jamnagar & Crude & 1,500 & 750 & 1 & 750 & 2,250 \\
Kochi & LNG & 150 & 75 & 1 & 75 & 225 \\
Mumbai & Crude & 800 & 400 & 1 & 400 & 1,200 \\
Mumbai & Fertilizer & 300 & 150 & 0 & 0 & 300 \\
Mumbai & LPG & 250 & 125 & 1 & 125 & 375 \\
Paradip & Fertilizer & 500 & 250 & 0 & 0 & 400 \\
Vizag & Crude & 400 & 200 & 1 & 200 & 600 \\
Vizag & LPG & 150 & 75 & 1 & 75 & 225 \\
\bottomrule
\end{tabular}
\end{table}

The medium-capex case activates expansion at seven predefined
port--commodity pairs and adds 1,775 normalised capacity units. Fixed and
variable construction costs are 6,850 and 1,370, respectively, so total
construction cost is 8,220 (14.14\% of total objective and 78.62\% of
first-stage cost). Fertilizer pairs are not expanded because fertilizer already
has complete service. The capex parameters are assumed stress-test inputs
rather than empirically calibrated investment costs. Under the high-capex
specification, 200 normalised capacity units are added across two
port--commodity pairs.

\end{document}